\documentclass{amsart}
\usepackage{amsfonts}
\usepackage{amsfonts,amsmath,amssymb,bm}
\usepackage{mathrsfs,mathtools,stmaryrd,wasysym}
\usepackage{enumerate}
\usepackage{isomath}
\usepackage{xspace,./mydef} 
\usepackage{hyperref}
\usepackage{xcolor}
\usepackage{accents}
\usepackage[numbers]{natbib}
\usepackage[printonlyused]{acronym}

\usepackage{pdfsync}

\newcommand{\WL}[1]{{\textcolor{black}{#1}}}

\newcommand{\TheTitle}{Numerical modeling of anisotropic ferroelectric materials with hybridizable discontinuous Galerkin methods}
\newcommand{\ShortTitle}{HDG approximation of the \ac{GLD} Model}

\acrodef{DG}{Discontinuous Galerkin Method}
\acrodef{HDG}{Hybridizable Discontinuous Galerkin Method}
\acrodef{FEM}{Finite Element Method}
\acrodef{BEM}{Boundary Element Method}
\acrodef{DPG}{Discontinuous Petrov-Galerkin Method}

\acrodef{FSI}{Fluid Structure Interaction}
\acrodef{ALE}{Arbitrary Lagrangian Eulerian}

\acrodef{IBM}{Immersed Boundary Method}
\acrodef{FE-IBM}{Finite Element Immersed Boundary Method}
\acrodef{IFEM}{Immersed Finite Element Method}

\acrodef{GCE}{geometry--convective explicit}

\acrodef{DD}{domain decomposition}
\acrodef{MG}{multigrid method}

\acrodef{GMRES}{generalized minimal residual method}
\acrodef{CG}{Conjugate Gradient Method}
\acrodef{PCG}{Preconditioned Conjugate Gradient Method}

\acrodef{FET}{field--effect--transistor}
\acrodef{MOSFET}{metal--oxide--semiconductor field--effect--transistor}
\acrodef{SS}{subthreshold swing}
\acrodef{NC}{negative capacitance}
\acrodef{GI}{gate insulator}
\acrodef{FE}{ferroelectric}
\acrodef{MFISFET}{metal--ferroelectric–-insulator–semiconductor FET}

\acrodef{GLD}{Ginzburg--Landau--Devonshire}

\begin{document}

\title[\ShortTitle]{{\TheTitle}}

\author{Wenyu Lei}
\address{School of Mathematical Sciences, University of Electronic Science and Technology of China, No.2006, Xiyuan Ave, West Hi-Tech Zone 611731 Chengdu, China}
\email{wenyu.lei@uestc.edu.cn}

\thanks{The authors would like to thank Dr. Luca Heltai and Dr. Milan Pe{\v{s}}i{\'{c}} for the usefull disscussions on the modeling as well as comments on this manuscript.}

\date{Draft version of \today.}

\begin{abstract}
	We investigate a gradient flow structure of the Ginzburg--Landau--Devonshire (GLD) model for anisotropic ferroelectric materials by reconstructing its energy form. We show that the modified energy form admits at least one minimizer. Under some regularity assumptions for the electric charge distribution and the initial polarization field, we prove that the $L^2$ gradient flow structure has a unique solution. To simulate the GLD model numerically, we propose an energy-stable semi-implicit time stepping scheme and an hybridizable discontinuous Galerkin method for space discretization. Some numerical tests are provided to verify the stability and convergence of the proposed numerical scheme as well as some properties of ferrorelectric materials.
\end{abstract}

\maketitle

%\tableofcontents

%\listoffixmes

%\pagebreak
\section{Introduction}
\label{s:intro}
\subsection{Motivation}
When designing a \ac{FET} using Metal-Oxide-Semiconductors (\acs{MOSFET}), one important task is to relax power density constraints in order to reduce the energy consumption of electronic devices.
A typical measurement used to assess the performance of a \acs{MOSFET} is given by the \ac{SS} value, i.e., the gate voltage value ($V_g$) which is necessary to vary the drain current ($I_d$) by ten times when it turns on or off.
In conventional \acs{MOSFET}s, where the conducting electrons are thermally injected over the potential barrier at the source/channel junction, the minimum theoretical \ac{SS} value is 60mV/decade at room temperature, which is known as the Boltzmann tyranny.
In most practical realizations, \acs{MOSFET}s have a SS value closer to 70--90mv/decade, due to the presence of a \ac{GI} that absorbs a non-negligible part of the applied gate voltage and require higher voltages to modulate the channel conductance.

A promising solution proposed by Salahuddin and Datta (see \cite{salahuddin2008use}) consists in using \ac{GI}s with \ac{NC} transforming the \ac{GI} into an internal voltage booster rather than a parasitic component.
Salahuddin and Datta suggested using the depoled state of a \ac{FE} thin film for such an internal voltage booster to decrease \ac{SS} far below the Boltzmann limit; we refer to \cite{park2019} for an introduction on the argument.
In contrast to \ac{MFISFET} memory devices, where the two distinguishable remnant polarization states of \ac{FE} thin films are used to represent two memory states, \ac{NC} mode operation requires a stable depoled state for the \ac{FE} thin film in order to sustain the \ac{NC} effect (see, e.g.,~\cite{hoffmann2018stabilization}).
As the MOSFET switching is a dynamic phenomenon, this has a significant implication for the low power operation because it allows a significantly decreased supply voltage for the device.
In this paper, we are interested in the mathematical modeling of the polarization field in an FE material as well as its numerical simulations.

\subsection{The \texorpdfstring{\ac{GLD}}{GLD} model}
A conventional FET at a constant temperature can be modeled using a Poisson equation for the potential $V$ and a charge continuity equations (or drift diffusion equations) for electric and hole densities \cite{markowichstationary}. The polarization field $\vecP=(P_1,\ldots,P_d)^{\tr}$ in the $d$-dimensional Euclidean space with $d\le 3$ is usually modeled as an explicit function of the electric field $\vecE = -\nabla V$. For example, in dielectric or insulating materials, $\vecP = \epsilon_0 \chi\vecE$, where $\epsilon_0$ is the permittivity of vacuum and $\chi$ is a function of the location. When considering \ac{FE} materials, the potential needs to be remodeled together with the polarization. Here we apply the Ginzburg--Landau--Devonshire model \cite{salahuddin2008use}: given a device occupying the material domain $\Omega$, we want to find $V$ and $\vecP$ that minimize the following \WL{Gibbs} free energy
\begin{equation}\label{e:energy-general}
  \widetilde G(V,\vecP) = \int_\Omega \bigg(\rho V + \vecP\cdot \nabla V -\frac{\epsilon_0\epsilon_b }2 |\nabla V|^2 + F_L(\vecP) + F_G(\{P_{i,j}\})\bigg)\diff x,
\end{equation}
where $\rho$ is the electric charge distribution, $\epsilon_b$ is the relative background permittivity of the \ac{FE} material and $g$ is the domain-wall coupling constant. $F_L(\vecP)$ is a Landau-type free-energy density functional and $F_G(\{P_{i,j}\})$ with $P_{i,j}=\partial_{x_j} P_i$ is a Ginzburg functional \cite{hlinka2006phenomenological, hu1997computer, nambu1994domain, gao2002domain}. For example, for the BaTiO${}_3$-type \ac{FE} materials, $F_L(\vecP)$ can be expanded with terms up to at least sixth order (cf. \cite{devonshire1949xcvi, hlinka2006phenomenological}), namely,
\[
  \begin{aligned}
    F_L(\vecP) & = \alpha_1(P_1^2+P_2^2+P_3^2) + \alpha_{11}(P_1^4+P_2^4+P_3^4)                                                \\
               & \qquad+\alpha_{12}(P_1^2P_2^2+P_2^2P_3^2+P_1^2P_3^2)+\alpha_{111}(P_1^6+P_2^6+P_3^6)                          \\
               & \qquad +\alpha_{112}[P_1^4(P_2^2+P_3^2)+P_2^4(P_1^2+P_3^2)+P_3^4(P_1^2+P_2^2)] +\alpha_{123}P_1^2P_2^2P_3^2 ,
  \end{aligned}
\]
and $F_G(\{P_{i,j}\})$ can be written as
\[
  \begin{aligned}
    F_G(\{P_{i,j}\}) & = \frac{G_{1,1}}{2}(P_{1,1}^2+P_{2,2}^2+P_{3,3}^2)                                        \\
                     & \qquad + G_{12}(P_{1,1}P_{2,2}+P_{2,2}P_{3,3}+P_{1,1}P_{3,3})                             \\
                     & \qquad + \frac{G_{4,4}}{2} [(P_{1,2}+P_{2,1})^2+(P_{2,3}+P_{3,2})^2+(P_{3,1}+P_{1,3})^2].
  \end{aligned}
\]
\subsubsection*{Simple anisotropic \ac{FE} materials}
Some materials (such as HfO${}_2$, cf. \cite{ikegami2019tcad}, see also \cite{saha2017ferroelectric,hoffmann2018stabilization}) may be modeled as anisotropic \ac{FE} materials, \ie the \ac{FE} \WL{behavior} happens mostly along a certain axis and dielectric \WL{behavior} is assumed for other axes. In such cases, we may model the total energy as
\begin{equation}\label{e:energy-component-wise}
  G(V,\vecP)  = \int_\Omega \bigg(\rho V + \vecP\cdot \nabla V -\frac{\epsilon} 2 |\nabla V|^2
  + \sum_{i=1}^d (F(P_i) + \frac{g_i}2 |\nabla P_i|^2) \bigg) \diff x .
\end{equation}
Here $F(P_i) := \alpha_i P_i^2 + \beta_i P_i^4 + \gamma_i P_i^6$, $\epsilon=\epsilon_0\epsilon_b$ and $\alpha_i,\beta_i,\gamma_i$ are \ac{FE} anisotropic constants. In terms of the FE anisotropic constants, from what follows, we assume that for each $i=1,\ldots,d$,
\begin{itemize}
  \item $P_i$ follows the dielectric property when $\alpha_i>0$ and $\beta_i,\gamma_i,g_i = 0$ or
  \item $P_i$ follows the \ac{FE} property when
        $\gamma_i,g_i>0$ (e.g. \cite{hoffmann2018stabilization}) or $\gamma_i=0$, $\beta_i,g_i>0$ (e.g. \cite{saha2017ferroelectric, ikegami2019tcad}).
\end{itemize}
Given a nanodevice occupying the region $\Omega$, we denote with $\Omega_{FE}\subset \Omega$ the domain of an \ac{FE} material. Following \cite{salahuddin2008use}, $\vecP$ vanishes in $\Omega\backslash \Omega_{FE}$. To simplify the treatment in this work, we set $\Omega_{FE} = \Omega$ so that the polarization field vanishes on $\partial\Omega$, the boundary of $\Omega$. In terms of boundary conditions for the potential, we set $\Gamma_D$ and $\Gamma_N$ to be disjoint subsets of $\partial\Omega$ such that $\Gamma_D\cup\Gamma_N = \partial\Omega$. We assume that $\Gamma_D\neq \emptyset$. For the simplicity of our discussion, we let the potential $V$ vanishes on $\Gamma_D$ and satisfying the following zero Neumann boundary condition on $\Gamma_N$:
\[
  \vecD\cdot \nu :=(\epsilon \vecE +\vecP)\cdot\nu = 0,\quad\text{ on }\Gamma_N .
\]
Here $\vecD$ is also referred to as the displacement field and $\nu$ denotes the outward normal vector. Noting that since $\vecP$ vanishes on $\partial\Omega$, the above Neumann boundary condition is also equivalent to
\[
  \vecE\cdot \nu = 0, \quad \text{ on } \Gamma_N .
\]

\subsection{Our contributions}
We shall search for a minimizer of the \WL{Gibbs} free energy $G(V,\vecP)$ incorporate with the boundary conditions provided in the previous subsection. According to \cite{lenarczyk2016physical,ikegami2019tcad} (see also \cite{penrose1990thermodynamically} for a general framework), we only search for the minimizer along the path for $P_i$ when $P_i$ follows the \ac{FE} property. More precisely speaking, we seek the solution $(V,\vecP)$ satisfying $\vecP(t=0) = \vecP_0$ and
\begin{equation}\label{e:system-strong}
  \begin{aligned}
    \delta_V G            & = 0,                    &  &                                                           \\
    \delta_{P_i} G        & = 0, \quad              &  & \text{ if } P_i \text{ follows  the dielectric property}, \\
    \rho_i \partial_t P_i & = -\delta_{P_i} G,\quad &  & \text{ if } P_i \text{ follows  the FE property} .
  \end{aligned}
\end{equation}
Here $\delta$ denotes the Fr\`echet derivative and $\rho_i>0$ is the viscosity constant. Motivated by \cite{albinus2002thermodynamic} (see also \cite{mielke2011gradient} for the drift-diffusion system), we provide a modified energy form in Proposition~\ref{p:gf-equivalent} such that the corresponding $L^2$-gradient flow coincides with the system \eqref{e:system-strong}. Such energy form will help us construct stable numerical schemes approximating the minimizer. We further show in Proposition~\ref{p:min-exist} that the minimization problem associated with the new energy form admits at least one minimizer.

We note that unlike the analysis for the classical Ginzburg--Landau system, or the Allen--Cahn equation (see e.g. \cite{bartels2011robust, feng2003numerical}), the polarization equation contains the coupling term $\GRAD V$, which is a vector-valued function of $\vecP$  associated with the Poisson equation $\delta_V G=0$. To show the well-posedness of the weak form of the system \eqref{e:system-strong}, we utilize the energy estimates in \cite{feng2003numerical} as well as the elliptic regularity for the Poisson equation. In Theorem~\ref{t:existence}, assuming that charge distribution $\rho$ is smooth enough, we prove that the weak solution exists if the initial polarization field is bounded in the energy space. Moreover, the weak solution is unique if $\Delta \vecP_0$ is bounded in the $L^2$ sense; see Theorem~\ref{t:uniqueness}.

We next simulate an \ac{FE} material by discretizing the problem \eqref{e:system-strong}. \WL{We shall first discretize the time derivative using the backward Euler method. In terms of computation of the nonlinear term in \eqref{e:system-strong} for each time step, perhaps a proper treatment is to linearize the term by the Newton's method (see e.g. \cite[Algorithm~6.1.(iii)]{bartels2015numerical}). Such technique may lead to the time-stepping method conditionally stable with respect to $\rho_i$ and $g$, namely we may require that the time step $\tau\le c\rho_i/g$ for some positive constant $c$. We may also need to guarantee that the discrete maximum principle holds, \ie at each time step, the $L^\infty$ norm of the polarization field is uniformly bounded; see e.g. Proposition~6.6 of \cite{bartels2015numerical}. So it is not easy in practice to tune the time stepping scheme when simulating for different ferroelectric materials.} In this work, we instead propose a semi-implicit scheme by splitting the gradient form provided in Proposition~\ref{p:gf-equivalent} with a convex and a concave part (see e.g. Chapter 6 of \cite{bartels2015numerical}). Such decomposition guarantees that the time-stepping method is unconditionally energy stable; see Proposition~\ref{p:stable}. \WL{Note that the nonlinear terms could be dominant during the evolution of the solutions. The fluxes for polarization field at the boundary could relatively large and thus create boundary layers. In order to capture these boundary layers in numerical simulations, a common strategy is to use discontinuous Galerkin methods to guarantee that (numerical) fluxes across faces of domain subdivisions are (weakly) continuous. Based on the time-stepping scheme, we consider a \ac{HDG} method which is originally proposed by \cite{cockburn2009unified}. One of advantages of \ac{HDG} schemes is that even though the degrees of freedom  are  defined on both cells and faces of a subdivision of the material domain, due to the hybridization (or static condensation) property, the degrees of freedom on cells \WL{can be eliminated} from the discrete system. Hence, the size of of discrete system to be solved for HDG discretization schemes is relatively smaller compared to regular discontinuous Galerkin schemes.}

The rest of the paper is organized as follows. In Section~\ref{s:min}, we provide a modified energy from for the weak form of \eqref{e:system-strong} and show that the corresponding minimization problem admits at least one minimizer. We show the existence and uniqueness of the weak solution of \eqref{e:system-strong} in Section~\ref{s:well-posedness}. Section~\ref{s:numerics} provides a semi-implicit time discretization and an HDG space discretization scheme, respectively. Using our proposed numerical method, we present some convergence tests as well as numerical tests in a monolayer device in Section~\ref{s:simulation}.

\section{A minimization problem}\label{s:min}
In this section we consider a minimization problem derived from the energy form \eqref{e:energy-component-wise} and show that at least one minimizer exists. Let us first introduce some notations. Let $\Omega$ be a bounded domain with Lipschitz boundary. Given an integer $k\ge 0$ and a real number $p>1$, we denote $W^{k,p}(\Omega)$ the standard Sobolev space and we let $H^{k}(\Omega) := W^{k,2}(\Omega)$. Let $\Hunz:=H^1_0(\Omega)$ be the subspace of $H^1(\Omega)$ whose functions vanishes at the boundary and
\[
  \mathcal V_D = \{w\in H^1(\Omega) : w = 0\text{ on }\Gamma_D\} .
\]
Finally, we denote $(\cdot,\cdot)$ and $\|\cdot\|$ the $L^2(\Omega)$ inner product and the $L^2(\Omega)$ norm, respectively. For two vector fields $\vecP,\vecQ\in [H^1(\Omega)]^d$, we set
\[
  (\GRAD\vecP,\GRAD\vecQ)=\sum_{i=1}^d(\GRAD P_i,\GRAD Q_i) .
\]

Consider the following minimization problem: find $V\in \Hunz_D$ and $\vecP\in [\Hunz]^d$ so that
\[
  G(V,\vecP) = \min_{\widetilde V\in \Hunz_D\text{ and }\widetilde\vecP\in[\Hunz]^d}
  G(\widetilde V,\widetilde\vecP).
\]
We shall search for a minimizer by solving its $L^2(\Omega)$ gradient flow. According to \cite{ikegami2019tcad, lenarczyk2016physical}, we additionally assume that the potential $V$ and certain components of the polarization field follow a dielectric \WL{behavior} relax instantly in time, \ie for the variational formulation, we have
\begin{equation}\label{e:poisson}
  \delta_V G[w] = -(\epsilon\GRAD V,\GRAD w) + (\rho-\DIV\vecP, w) = 0, \quad\forall w\in \Hunz_D
\end{equation}
and
\begin{equation}\label{e:dielectric}
  \delta_{P_i}G[Q_i] = 2\alpha_i (P_i, Q_i) + (\partial_{x_i}V , Q_i) = 0,\quad \forall Q_i\in \Hunz .
\end{equation}
We also point out that we have applied integration by parts $(\vecP,\GRAD w) = -(\DIV\vecP,w)$ in \eqref{e:poisson}. For $P_i$ that follows the \ac{FE} \WL{behavior}, we have that
\[
  \rho_i\partial_t P_i = - \delta_{P_i}G[Q_i] ,\quad \forall Q_i \in \mathcal V,
\]
or
\begin{equation}\label{e:ferroelectric}
  \rho_i(\partial_t P_i, Q_i) + g_i(\GRAD P_i,Q_i) +(DF(P_i)P_i,Q_i)
  + (\partial_{x_i}V,Q_i) = 0 .
\end{equation}
Here $DF(P_i)=2\alpha + 4\beta P_i^2+6\gamma P_i^4$. The above equation is also referred to as the Landau-Khalatnikov equation \cite{landau1954anomalous}.

\subsection{A modified energy form}
We shall reconstruct an energy form associated with the system \eqref{e:poisson}--\eqref{e:ferroelectric}. Given a polarization field $\vecP\in [\Hunz]^d$, we consider the following energy functional
\begin{equation}\label{e:energy}
  I(\vecP) =\int_\Omega \bigg(\frac{\epsilon}{2}|\GRAD\Vp|^2 + F(\vecP) + \sum_{i=1}^d \frac{g_i}2 |\nabla P_i|^2 \bigg)\diff x,
\end{equation}
where $F(\vecP)=\sum_{i=1}^d F(P_i)$ and $\Vp:=\Phi(\rho-\DIV\vecP)$ with $\Phi : L^2(\Omega)\to \Hunz$ denoting the solution operator so that for $g\in L^2(\Omega)$, $\Phi g$ uniquely solves
\begin{equation}\label{e:Phi}
  (\epsilon\GRAD \Phi g,\GRAD w) = (g,w),\quad\forall w\in \Hunz_D .
\end{equation}
So $V = \Vp$ in \eqref{e:poisson}.
\begin{proposition}\label{p:gf-equivalent}
  We have $I(\vecP) = G(\Vp, \vecP)$ and the corresponding $L^2(\Omega)$ gradient flow coincides with \eqref{e:poisson}--\eqref{e:ferroelectric} by setting $V=\Vp$. In particular,
  \begin{equation}\label{e:gf-alternative}
    \rho_{i}\partial_t P_i = -\delta_{P_i} I(\vecP) ,
  \end{equation}
  with $\rho_{i} = 0$ if $P_i$ follows the dielectric property and $\rho_i>0$ if $P_i$ follows the \ac{FE} property.
\end{proposition}
\begin{proof}
  Choosing $w=\Vp$ in \eqref{e:poisson} we get
  \[
    \int_\Omega\frac\epsilon2 |\GRAD\Vp|^2 \diff x
    = \int_\Omega \rho \Vp + \vecP\cdot \GRAD \Vp - \frac\epsilon2 |\GRAD\Vp|^2 \diff x.
  \]
  This implies that $I(\vecP) = G(\Vp, \vecP)$ according to \eqref{e:energy-component-wise}.
  To obtain the gradient flow from $I(\vecP)$, it suffices to show that the differential of $\vecP\cdot\GRAD\Vp$ is $\GRAD \Vp$. Letting $H(g)=\tfrac\epsilon2\int_\Omega|\GRAD\Phi g|^2\diff x$, Lemma~6.1 of \cite{albinus2002thermodynamic} implies that $\delta_g H(g)[w] = (\Phi g, w)$, with $w\in \Hunz$. So  given $\vecQ\in [C^\infty_0(\Omega)]^d$,
  \[
    \begin{aligned}
      H(\rho-\DIV(\vecP+\vecQ)) & -H(\rho-\DIV\vecP)
      \\
                                & =
      \delta_g H(\rho-\DIV\vecP)[-\DIV\vecQ] + O(\|\DIV\vecQ\|)           \\
                                & = -(\Vp,\DIV\vecQ) + O(\|\DIV\vecQ\|) .
    \end{aligned}
  \]
  Using integration by parts for the first term on the right hand side above, we conclude that $\delta_\vecP H(\rho-\DIV\vecP)[\vecQ] = (\GRAD\Vp, \vecQ)$, as desired.
\end{proof}

\subsection{Existence of minimizers}
We shall apply the direct method in the calculus of variation to show that the following minimization problem
\begin{equation}\label{e:min-I}
  I(\vecP)  = \min_{\tilde\vecP\in [\Hunz]^d} I(\tilde\vecP)
\end{equation}
has at least one minimizer.

\subsubsection*{Weakly lower semicontinuity for the higher order terms}
Let us first verify that
\[
  \int_\Omega \bigg(\frac{\epsilon}{2}|\GRAD\Vp|^2 + \sum_{i=1}^d \frac{g_i}2 |\nabla P_i|^2 \bigg)\diff x
\]
is weakly lower semicontinuous and $I(\vecP)$ is coercive.
It is well known that the second term above is weakly lower semicontinuous (see e.g., \cite[Example~2.1]{bartels2015numerical}). Here we check for $\int_\Omega \frac\epsilon2|\GRAD V_{\vecP}|^2$. Let $\{\vecP_n\}_{n\in\mathbb N}$ be a sequence in $[\Hunz]^d$ such that $\vecP_n$ weakly converges to $\vecP$ in $[\Hunz]^d$. Note that
\[
  \begin{aligned}
     & \int_\Omega \frac\epsilon2 |\GRAD V_{\vecP}|^2\diff x - \int_\Omega \frac\epsilon2 |\GRAD V_{\vecP_n}|^2\diff x
    +\int_\Omega \frac\epsilon2 |\GRAD (\Vp-V_{\vecP_n})|^2\diff x                                                                               \\
     & \quad\quad\quad\quad= \int_\Omega \epsilon \GRAD \Vp\cdot\GRAD(\Vp-V_{\vecP_n})\diff x= \int_\Omega (\vecP-\vecP_n)\cdot\GRAD\Vp \diff x,
  \end{aligned}
\]
where we applied \eqref{e:poisson} for both $\Vp$ and $V_{\vecP_n}$ in the last equality above. Letting $n\to\infty$ yields
\[
  \int_\Omega \frac \epsilon2|\GRAD\Vp|^2\diff x\le\liminf_{n\to\infty}\int_\Omega \frac \epsilon2|\GRAD V_{\vecP_n}|^2\diff x .
\]

\subsubsection*{Coercivity}
Note for both properties of $P_i$, the polynomial $F(P_i)$ has a minimum $-c_0$ with $c_0\ge 0$. Suppose there is a sequence $\{\vecP_n\}_{n=1}^\infty\subset [\Hunz]^d$ with $\|\vecP_n\|_{H^1(\Omega)}\to \infty$. We have
\begin{equation}\label{i:lb}
  \int_\Omega F(\vecP_n)\diff x \ge -d c_0 |\Omega| .
\end{equation}
Whence,
\begin{equation}\label{i:coercive}
  I(\vecP_n) \ge \frac{g}{2}\|\GRAD \vecP_n\|^2 -d c_0 |\Omega|\to \infty,\quad\text{ as } n\to \infty
\end{equation}
as $n\to\infty$. This implies that $I(\vecP)$ is coercive.

\subsubsection*{Existence of the minimizer}
Now we are ready to show the existence of a minimizer for the problem \eqref{e:min-I}.
\begin{proposition}[existence of a minimizer]\label{p:min-exist}
  Given an \ac{FE} material occupying the domain $\Omega$, let $c_0$ be defined as in \eqref{i:lb}. Then there exists at least one minimizer for problem \eqref{e:min-I}.
\end{proposition}
\begin{proof}
  In view of \eqref{i:coercive}, $I(\vecP)$ is also bounded below. So there exists a sequence $\{\vecP_n\}_{n=1}^\infty\subset [\Hunz]^d$ such that $\lim_{n\to\infty}I(\vecP_n)=\inf_{\widetilde\vecP\in[\Hunz]^d}I(\widetilde\vecP)$. The coercivity of $I(\vecP)$ (cf. \eqref{i:coercive}) implies that $\vecP_n$ is bounded in $[\Hunz]^d$ so there exists a subsequence of $\{\vecP_{n}\}$ weakly converging to some limit $\vecP$ in $[\Hunz]^d$. Here without loss of generality, we denote this subsequence by $\{\vecP_n\}$. Due to the lower semicontinuity of the higher order term in $I(\vecP)$, we have
  \[
    \int_\Omega \frac \epsilon2|\GRAD\Vp|^2 + \sum_{i=1}^d \frac{g_i}2 |\nabla P_i|^2 \diff x\le
    \liminf_{n\to\infty} 	\int_\Omega \frac \epsilon2|\GRAD V_{\vecP_{n}}|^2 + \sum_{i=1}^d \frac{g_i}2 |\nabla P_{ni}|^2 \diff x,
  \]
  with $P_{ni}$ denoting the $i$-th component of $\vecP_n$. On the other hand, by the compactly embedding from $\Hunz$ into $L^6(\Omega)$ for $d=1,2,3$, $\{\vecP_{n}\}$ (again, w.l.o.g. by passing to a subsequence) also strongly converges to $\vecP$ in $L^6(\Omega)$. So
  \[
    \int_\Omega F(\vecP_{n}) \diff x \to \int_\Omega F(\vecP) \diff x
    ,\quad\text{ as } n\to\infty .
  \]
  Combing the above two limits implies that
  \[
    I(\vecP) \le \liminf_{n\to \infty} I(\vecP_{n}) = \lim_{n\to\infty} I(\vecP_{n}) = \inf_{\widetilde\vecP\in[\Hunz]^d} I(\widetilde\vecP).
  \]
  Since $I(\vecP) \ge  \inf_{\widetilde\vecP\in[\Hunz]^d} I(\widetilde\vecP)$ by definition, we conclude that $\vecP$ is a minimizer of the problem \eqref{e:min-I}.
\end{proof}

\begin{remark}[(uniqueness)]\label{r:unique}
  It is sufficient to show that there exists at most one minimizer when $I(P)$ is strictly convex. We  note that $\int_\Omega |\GRAD\Vp|^2$ is convex with respect to $\vecP$. Indeed, as $\int_\Omega|\GRAD V|^2$ is convex for $V$, we have for $\vecP_1,\vecP_2\in [\Hunz]^d$ and $\lambda\in[0,1]$,
  \[
    \begin{aligned}
      \lambda \int_\Omega |\GRAD V_{\vecP_1}|^2\diff x & + (1-\lambda) \int_\Omega |\GRAD V_{\vecP_2}|^2\diff x                       \\
                                                       & \ge \int_\Omega |\GRAD(\lambda V_{\vecP_1}+(1-\lambda)V_{\vecP_2})|^2\diff x \\	                                                                                                        & =\int_\Omega |\GRAD \Phi(\rho - \DIV(\lambda\vecP_1+(1-\lambda)\vecP_2))|^2\diff x \\
                                                       & =\int_\Omega |\GRAD V_{\lambda \vecP_1+(1-\lambda)V_{\vecP_2}}|^2\diff x,
    \end{aligned}
  \]
  recalling that the linear operator $\Phi$ is defined in \eqref{e:Phi}.
  However, we cannot guarantee the convexity due to choice of the constants $\alpha_i,\beta_i,\gamma_i$ inside the Landau-type density functional $\int_\Omega F(P_i)$. Considering the second derivative of $F(P_i)$, a sufficient condition for uniqueness is that for each component $P_i$ following the \ac{FE} property, the quadratic form $p(t) = 30\gamma_i t^2 + 12\beta_i t + 2\alpha_i$ is strictly positive for $t>0$. This leads to the following two conditions:
  \begin{itemize}
    \item $\alpha_i,\beta_i>0$ or
    \item $\gamma_i>0$, $\beta_i <0$ and $3\beta_i^2/(5\gamma_i) < \alpha_i$.
  \end{itemize}
\end{remark}

\section{Well-posedness of the gradient flow}\label{s:well-posedness}
This section is devoted to showing the existence and uniqueness of the weak solution to the $L^2$ gradient flow \eqref{e:poisson}--\eqref{e:ferroelectric}. For simplicity, we assume that all components of the polarization field follow the \ac{FE} property with the same constants denoting by $\rho,\alpha,\beta,\gamma$ and $g$. To discuss the well-posedness, we shall introduce the Bochner-Sobolev spaces. Given a Banach space $B$, $r\ge 1$ and a final  time $T$, the space $L^r(0,T; B)$ is a collection of functions so that for each $w\in L^r(0,T; B)$, the norm
\[
  \|w\|_{L^r(0,T;B)} =
  \left\{
  \begin{aligned}
     & \bigg(\int_0^T\|w\|_B^r \diff t\bigg)^{1/r}, &  & \quad r\in[1,\infty), \\
     & \text{esssup}_{t\in[0,T]}\|w\|_{B},          &  & \quad r=\infty ,
  \end{aligned}
  \right .
\]
is finite. We also define the space $H^1(0,T;B)$ by
\[
  H^1(0,T;B) :=\{u\in L^2(0,T;B) : \partial_t u \in L^2(0,T;B)\} .
\]
Denote $\Hunz'$ the dual space of $\Hunz$ and let $\langle \cdot,\cdot\rangle$ the corresponding duality pairing. A weak formulation of problem \eqref{e:poisson}--\eqref{e:ferroelectric} reads: given an initial polarization $\vecP_0$ and a final time $T$, we want to find $\vecP:=\vecP(t)\in L^2(0,T; [\Hunz]^d)\cap H^1(0,T;[\Hunz']^d)$ such that $\vecP(0) = \vecP_0$ and
\begin{equation}\label{e:gf-l2}
  \begin{aligned}
     & \rho_v\langle\partial_t \vecP, \vecQ\rangle + g(\GRAD\vecP,\GRAD\vecQ)                                                \\
     & \quad\quad+ (\mathbf DF(\vecP),\vecQ) + (\GRAD\Vp, \vecQ) = 0,\quad\forall \vecQ\in[\Hunz]^d \text{ and } t\in (0,T],
  \end{aligned}
\end{equation}
where $\Vp$ is given by \eqref{e:energy} and $\mathbf DF(\vecP)=(DF(P_i) P_i)_{i=1}^d$. Here we also note that for simplicity of our discussion, we set the viscosity constant $\rho_v = \rho_i$ for $i=1,\ldots,d$.

\subsection{Existence}
Thanks to Fredholm alternative, there exist $L^2(\Omega)$-orthonormal eigenfunctions $\{\phi_i\}_{i=1}^\infty$ with increasing eigenvalues $\{\lambda_i\}_{i=1}^\infty\subset (0,\infty)$ satisfying
\[
  (\GRAD \phi_i, \GRAD w) = \lambda_i (\phi_i, w),\quad \forall w\in \Hunz.
\]
Given a positive integer $N$, set $X_N = \text{span}\{\phi_i\}_{i=1}^N$ and we consider the following finite dimensional problem: find $\vecP_N(t)\in [X_N]^d$ such that
\begin{equation}\label{e:gf-l2-n}
  \begin{aligned}
    \rho_v & \langle  \partial_t  \vecP_N  ,\vecQ\rangle  + g(\GRAD\vecP_N,\GRAD\vecQ) &     \\
           & +(\mathbf DF(\vecP_N),\vecQ) +(\GRAD V_{\vecP_N},\vecQ) = 0,              &   &
    \forall \vecQ\in [X_N]^d \text{ and } t \in (0, T],                                      \\
           & (\vecP_N(0),\vecQ) = (\vecP_0,\vecQ),                                     &   &
    \forall \vecQ\in [X_N]^d .
  \end{aligned}
\end{equation}
We first note that if the charge distribution $\rho\in L^\infty(0,T; L^2(\Omega))$, the above problem has a solution. This can be shown by exploiting the fact that $\mathbf DF$ is a bounded operator in $[X_N]^d$ equipped with $[L^2(\Omega)]^d$ norm. Here we sketch one approach of the proof: one can derive that $\delta_{\vecP}I(\vecP)[\vecQ]$ is bounded and semicoercive, namely there exist       positive constants $C,C_1,C_2,C_3$ so that
\[
  \delta_{\vecP}I(\vecP_N)[\vecQ] \le C(\|\vecP_N\|+\|\GRAD\vecP_N\|+\|\rho\|)\|\vecQ\|_{[\Hunz]^d}
\]
and
\[
  \delta_{\vecP}I(\vecP_N)[\vecP_N] \ge C_1\|\GRAD\vecP_N\|^2-C_2\|\vecP_N\|^2-C_3\|\rho\|^2.
\]
Since $I(\vecP)$ is semicoercive and weakly lower semicontinuous, we can follow the argument from \cite[Section 2.3]{bartels2015numerical} to show the existence of $\vecP_N$. That is constructing a solution of $\vecP_N$ by an approximation using the backward Euler scheme.

We next provide some auxiliary estimates related to $\vecP_N$. From what follows, we use $C$ to denote a generic constant independent of $N$.
\begin{lemma}\label{l:estimates}
  Suppose that the initial polarization field $\vecP_0\in [H^1(\Omega)\cap L^\infty(\Omega)]^d$ and the charge distribution $\rho\in H^1(0,T;L^2(\Omega))$. Then for a fixed final time $T$, there holds:
  \begin{enumerate}[(a)]
    \item $\vecP_N\in L^\infty(0,T;[L^2(\Omega)]^d) \cap L^2(0,T;[\Hunz]^d)$ and $V_{\vecP_N}\in L^2(0,T;\Hunz_D)$, \ie there exists a positive constant $C$ independent of $N$ satisfying
          \[
            \max_{t\in[0,T]}\|\vecP_N(t)\|^2 +\int_0^T \|\GRAD \vecP_N(t)\|^2 + \|\vecP_N(t)\|^2 + \|\GRAD V_{\vecP_N}(t)\|^2 \diff t \le C .
          \]
    \item Assume that we have the $H^2(\Omega)$ elliptic regularity in \eqref{e:Phi}, \ie for $g\in L^2(\Omega)$,
          \[
            \|\Phi g\|_{H^2(\Omega)} \le C\|g\|,
          \]
          where the constant $C$ only depends on $d$ and $\Omega$. Then $\vecP_N\in H^1(0,T; L^2(\Omega))$. More precisely,
          \[
            I(\vecP_N(T))+\rho_v\int_0^T \|\partial_t\vecP_N(t)\|^2 \diff t
            \le I(\vecP_0) +\frac12\int_0^T \|\rho_t(t)\|^2 +\|V_{\vecP_N}(t)\|^2\diff t.
          \]
    \item There exists a positive constant $C$ independent of $N$ such that
          \[
            \int_0^T\|\Delta \vecP_N(t)\|^2 \diff t \le C\int_0^T \|\GRAD \vecP_N(t)\|^2 + \|\partial_t\vecP_N(t)\|^2 + \|\rho(t)\|^2\diff t .
          \]
          Under the assumption in part (b), we conclude that $\vecP_N$ is uniformly bounded in $L^2(0,T; H^2(\Omega))$.
  \end{enumerate}
\end{lemma}
\begin{proof}
  \boxed{1}  To prove (a), we let $\vecQ = \vecP_N$ in \eqref{e:gf-l2-n} to get that for any $\mu>0$,
  \[
    \begin{aligned}
      \frac{\rho_v}2\diff_t\|\vecP_N(t)\|^2 & + g\|\GRAD\vecP_N(t)\|^2  +\epsilon \|\GRAD V_{\vecP_N}(t)\|^2                                  \\
                                            & = (\rho,V_{\vecP_N}) + c_1\|\vecP_N(t)\|^2                                                      \\
                                            & \le \frac{\mu C_p^2}{2}\|\GRAD V_{\vecP_N}\|^2+\frac{1}{2\mu}\|\rho\|^2 + c_1\|\vecP_N(t)\|^2 ,
    \end{aligned}
  \]
  where we used the facts that $DF(P_i) = 2\alpha+4\beta P_i^2 + 6\gamma P_i^4\ge -c_1$ for some $c_1\ge 0$ and $(\GRAD V_{\vecP_N},\vecP_N)= \epsilon\|\GRAD V_{\vecP_N}\|^2-(\rho,V_{\vecP_N})$ according to \eqref{e:poisson}. We also note that for the last inequality above, we applied the Schwarz inequality as well as the Poincar\'e inequality with $C_p$ denoting the Poincar\'e constant. Now we set $\mu =\epsilon/C_p^2$ and integrate the above inequality over $[0,T]$ to obtain that
  \[
    \begin{aligned}
      \|\vecP_N(T)\|^2 & \le \|\vecP_N(T)\|^2+ \frac{2}{\rho_v}\int_0^T g\|\GRAD \vecP_N(t)\|^2 +
      \frac\epsilon2\|\GRAD V_{\vecP_N}\|^2\diff t                                                                                       \\
                       & \le \|\vecP(0)\|^2+\frac{2}{\rho_v}\int_0^T c_1\|\vecP_N(t)\|^2 + \frac{C_p^2}{2\epsilon}\|\rho(t)\|^2\diff t .
    \end{aligned}
  \]
  Gr\"{o}wall's inequality implies that $\max_{t\in[0,T]}\|\vecP_N(t)\|^2\le C$. Applying this estimate to the second inequality above implies that
  \[
    \int_0^T \|\GRAD \vecP_N(t)\|^2 + \|\GRAD V_{\vecP_N}(t)\|^2 \diff t \le C.
  \]
  Then,
  \begin{equation}\label{i:poincare}
    \int_0^T \|\vecP_N(t)\|^2 \diff t \le C\int_0^T \|\GRAD \vecP_N(t)\|^2\diff t \le C
  \end{equation}
  due to the Poincar\'e inequality.

  \boxed{2} Before showing the estimate (b), we note that $V_{\vecP_N}\in H^2(\Omega)$ due to the elliptic regularity assumption and hence $V_{\vecP_N}$ satisfies the equation
  \[
    -\epsilon\Delta V_{\vecP_N} = \rho-\DIV\vecP_N
  \]
  in the $L^2(\Omega)$ sense. Differentiate the above equation with respect to $t$ and multiply by a test function $w\in\Hunz_D$ to obtain
  \begin{equation}\label{e:poisson-t}
    \epsilon(\partial_t(\GRAD V_{\vecP_N}),\GRAD w) = (\partial_t\vecP_N,\GRAD w) + (\partial_t\rho, w).
  \end{equation}
  Now we choose $\vecQ = \partial_t \vecP_N\in [\Hunz]^d$ in \eqref{e:gf-l2-n} and apply \eqref{e:poisson-t} with $w=V_{\vecP_N}\in \Hunz_D$ to write
  \[
    \begin{aligned}
      \rho_v\|\partial_t \vecP_N(t)\|^2 & + \frac g2 \diff_t\|\GRAD \vecP_N(t)\|^2+\diff_t\int_\Omega F(\vecP_N(t))\diff x + \frac\epsilon 2\diff_t\|\GRAD V_{\vecP_N}(t)\|^2 \\
                                        & =\rho_v\|\partial_t \vecP_N(t)\|^2 + \diff_t I(\vecP_N(t))= (\partial_t\rho,V_{\vecP_N})
      \le \frac12\|\partial_t \rho\|^2+\frac12\|V_{\vecP_N}\|^2.
    \end{aligned}
  \]
  The assertion follows by integrating of the above inequality over $[0,T]$ and by the fact that $I(\vecP_0)$ is finite.

  \boxed{3} In terms of the estimate (c), we set $\vecQ = -\Delta \vecP_N\in [X_N]^d$ in \eqref{e:gf-l2-n}. Note that for $i=1,\ldots,d$, integration by parts yields that
  \[
    -(DF(P_i)P_i, \Delta P_i) = \int_\Omega(2\alpha+12\beta P_i^2 + 30\gamma P_i^4)|\GRAD P_i|^2\diff x \ge -c_2 \|\GRAD P_i\|^2,
  \]
  with some $c_2\ge 0$.
  Applying this estimate in \eqref{e:gf-l2-n} to get
  \[
    \begin{aligned}
      g\|\Delta\vecP_N\|^2 & \le \rho_v(\partial_t\vecP_N,\Delta\vecP_N)+c_2\|\GRAD\vecP_N\|^2 +(\GRAD V_{\vecP_N},\Delta\vecP_N)          \\
                           & \le c_2\|\GRAD\vecP_N\|^2+ \frac{\rho_v\eta}2 \|\Delta\vecP_N\|^2+\frac{\rho_v}{2\eta}\|\partial_t\vecP_N\|^2
      +\frac{\mu}{2}\|\Delta\vecP_N\|^2 + \frac{1}{2\mu}\|\GRAD V_{\vecP_N}\|
    \end{aligned}
  \]
  with arbitrary positive $\eta$ and $\mu$. Letting $\eta$ and $\mu$ small enough so that $g-\rho_v\eta/2-\mu/2>0$ to guarantee that
  \begin{equation}\label{i:h2-estimate}
    \begin{aligned}
      \|\Delta\vecP_N\|^2 & \le C(\|\GRAD\vecP_N\|^2 + \|\partial_t\vecP_N\|^2 + \|\GRAD V_{\vecP_N}\|^2)     \\
                          & \le C(\|\GRAD\vecP_N\|^2 + \|\partial_t\vecP_N\|^2 + \|\vecP_N\|^2 + \|\rho\|^2),
    \end{aligned}
  \end{equation}
  where for the last inequality above we note that $\|\GRAD V_{\vecP_N}\| \le C(\|\vecP_N\| + \|\rho\|)$ according to \eqref{e:poisson}.
  The assertion then directly follows from the integration of the above inequality over $[0,T]$ and \eqref{i:poincare}.
\end{proof}

Now we are in a position to show the existence of weak solution to the $L^2(\Omega)$ gradient flow \eqref{e:gf-l2}.
\begin{theorem}[existence of the $L^2(\Omega)$ gradient flow]\label{t:existence}
  Suppose that the assumptions in Lemma~\ref{l:estimates} hold. Then here exists a polarization field
  \[
    \vecP \in L^2(0,T; [\Hunz\cap H^2(\Omega)]^d)\cap H^1(0,T; [L^2(\Omega)]^d).
  \]
  solves \eqref{e:gf-l2}.
\end{theorem}
\begin{proof}
  We summarize the results in Lemma~\ref{l:estimates} to get
  \[
    \|\vecP_N(t)\|_{L^\infty(0,T;L^2(\Omega))} + \|\vecP_N\|_{L^2(0,T;\Hunz\cap H^2(\Omega))} + \|\vecP_N\|_{H^1(0,T;L^2(\Omega))} \le C,
  \]
  where the constant $C$ is independent of $N$. By Sobolev embedding, we also have $\|\vecP_N\|_{L^2(0,T;L^\infty(\Omega))}\le C$. This implies that
  \[
    \|\mathbf DF(\vecP_N)\|_{L^2(0,T;L^2(\Omega))}\le C
  \]
  with the constant $C$ independent of $N$. Therefore, there exists a vector-valued function $\bm \xi\in L^2(0,T;[L^2(\Omega)]^d)$, a vector-valued function $\vecP\in L^2(0,T;[\Hunz\cap H^2(\Omega)]^d)\cap H^1(0,T;[L^2(\Omega)]^d) $ and a subsequence of $\{\vecP_N\}$, which is denoted by $\{\vecP_N\}$, satisfying
  \[
    \begin{aligned}
      \vecP_N \to \vecP\quad                     & \text{ weakly in } L^2(0,T;[\Hunz\cap H^2(\Omega)]^d), \\
      \partial_t\vecP_N \to \partial_t\vecP\quad & \text{ weakly in } L^2(0,T;[L^2(\Omega)]^d),           \\
      \vecP_N \to  \vecP\quad                    & \text{ weak-* in } L^\infty(0,T;[L^2(\Omega)]^d),      \\\
      \mathbf DF(\vecP_N) \to \bm\xi\quad        & \text{ weakly in } L^2(0,T;[L^2(\Omega)]^d) .
    \end{aligned}
  \]
  According to the compact embedding
  \[
    L^2(0,T;\Hunz)\cap H^1(0,T;\Hunz')\to L^2(0,T;L^2(\Omega)),
  \]
  we can extract a subsequence from $\{\vecP_N\}$ (denoted again by $\{\vecP_N\}$) so that $\vecP_N$ converges to $\vecP$ a.e. in $\Omega\times[0,T]$. By the generalized dominated convergence theorem,
  \[
    \int_0^T(\mathbf DF(\vecP_N),\vecQ)\diff t \to \int_0^T (\mathbf  DF(\vecP),\vecQ)\diff t,\quad\text{ as } N\to \infty ,
  \]
  which implies that $\bm\xi = \mathbf DF(\vecP)$. We also note that $\GRAD V_{\vecP_N}$ converges to $\GRAD\Vp$ strongly in $L^2(0,T;L^2(\Omega))$ in view of \eqref{e:poisson}. So
  \[
    \|\GRAD(V_{\vecP_N}-\Vp)\| \le \frac1\epsilon \|\vecP_N-\vecP\| .
  \]
  Due to the compact embedding
  \[
    L^2(0,T;\Hunz\cap H^2(\Omega))\cap H^1(0,T;L^2(\Omega)) \to C([0,T];L^2(\Omega)),
  \]
  $\vecP_N(0)$ converges to $\vecP_0$ in $L^2(\Omega)$. Now we apply all the limits above in \eqref{e:gf-l2-n} to conclude that $\vecP$ satisfies equation \eqref{e:gf-l2} and that $\vecP\in L^2(0,T;\Hunz\cap H^2(\Omega))\cap H^1(0,T;L^2(\Omega))$.
\end{proof}

\subsection{Uniqueness}
To show the uniqueness of the $L^2(\Omega)$ gradient flow, we further assume that $\|\Delta\vecP_0\|\le\infty$. This leads to the following regularity property for the solution $\vecP$.
\begin{lemma}\label{l:max}
  Under the assumptions in Lemma~\ref{l:estimates} and $\|\Delta\vecP_0\|\le\infty$,
  we have that if $\vecP$ is a weak solution to \eqref{e:gf-l2}, then
  \[
    \max_{t\in[0,T]} \bigg(\|\partial_t\vecP\| + \|\vecP\|_{[H^2(\Omega)]^d}\bigg) \le C .
  \]
  Hence, $\vecP\in L^\infty(\Omega\times[0,T])$ by the Sobolev embedding.
\end{lemma}
\begin{proof}
  \boxed{1} We first note that $\Vp$ and $\vecP$ satisfies the following system in the $L^2(\Omega)$ sense:
  \begin{equation}\label{e:strong}
    \begin{aligned}
      \rho_v\partial_t\vecP -g\Delta \vecP + \mathbf DF(\vecP) + \GRAD\Vp & = 0,                                  \\
      \quad\text{ with } \DIV(-\epsilon\GRAD \Vp + \vecP)                 & = \rho \text{ and } \vecP(0)=\vecP_0.
    \end{aligned}
  \end{equation}
  Similar to \eqref{e:poisson-t}, we differentiate the above equations with respect to $t$ and multiply with test functions to get
  \[
    \begin{aligned}
       & \epsilon(\GRAD(\partial_t \Vp),\GRAD w) = (\partial_t\vecP,\GRAD w) + (\partial_t\rho, w), & \quad\forall w\in \Hunz_D,             \\
       & \rho_v(\partial_{tt}\vecP,\GRAD\vecQ) + g(\GRAD(\partial_t\vecP),\GRAD\vecQ)                                                        \\
       & \quad\quad\quad + (\partial_t \mathbf DF(\vecP),Q) + (\GRAD(\partial_t\Vp),\vecQ) = 0,     & \quad\quad\forall \vecQ\in [\Hunz]^d .
    \end{aligned}
  \]
  Letting $w=\partial_t\Vp$ and $\vecQ = \partial_t\vecP$, we obtain that for any $\mu>0$,
  \[
    \begin{aligned}
      \frac{\rho_v}2\diff_t & \|\partial_t\vecP\|^2 + g\|\GRAD(\partial_t \vecP)\|^2
      + \epsilon\|\GRAD(\partial_t\Vp)\|^2                                                                                                            \\
                            & \le c_2 \|\partial_t\vecP\|^2 + (\partial_t\rho, \partial_t\Vp)                                                         \\
                            & \le c_2 \|\partial_t\vecP\|^2 + \frac{1}{2\mu} \|\partial_t \rho\|^2 + \frac{\mu C_p^2}{2} \|\GRAD(\partial_t\Vp)\|^2 , \\
    \end{aligned}
  \]
  recalling from Step~1 in the proof of Lemma~\ref{l:estimates} that $C_p$ is the Poincar\'e constant. We set $\mu =\epsilon/C_p^2$ to yield
  \begin{equation}\label{i:h22-estimate}
    \frac{\rho_v}2\diff_t  \|\partial_t\vecP\|^2 + g\|\GRAD(\partial_t \vecP)\|^2
    + \frac{\epsilon}{2}\|\GRAD(\partial_t\Vp)\| \le c_2 \|\partial_t\vecP\|^2 + \frac{C_p^2}{2\epsilon} \|\partial_t \rho\|^2 .
  \end{equation}
  Since $\vecP\in H^1(0,T;[L^2(\Omega)]^d)$, $\rho\in H^1(0,T; L^2(\Omega))$ and
  \[
  \begin{aligned}
    \lim_{t\to0}\|\partial_t\vecP_0\| = \tfrac{1}{\rho_v}\|&g\Delta\vecP_0-\mathbf DF(\vecP_0)-\GRAD V_{\vecP_0}\| \\
    &\le C(\|\Delta\vecP_0\| + \|\mathbf DF(\vecP_0)\| + \|\vecP_0\| + \|\rho(t=0)\|) \le C,
   \end{aligned}
  \]
  we integrate \eqref{i:h22-estimate} over $[0,T]$ to derive that $\partial_t\vecP\in L^\infty(0,T; [L^2(\Omega)]^d)$.

  \boxed{2} We multiply the term $-\Delta\vecP$ on both sides of the first equation in \eqref{e:strong} and apply integration by parts for the term $-(\partial_t\vecP,\Delta\vecP)$ to write
  \[
    \frac{\rho_v}2\diff_t \|\GRAD\vecP\|^2
    +g\|\Delta\vecP\|^2 \le c_2\|\GRAD\vecP\|^2 +(\GRAD V_{\vecP},\Delta\vecP)
  \]
  Following the argument for \eqref{i:h2-estimate}, we have
  \[
    \frac{\rho_v}2\diff_t \|\GRAD\vecP\|^2
    +g\|\Delta\vecP\|^2 \le C(\|\GRAD\vecP\|^2 + \|\rho\|^2)
  \]
  Integrating the above inequality over $[0,T]$ implies that $\vecP\in L^\infty(0,T;[\Hunz]^d)$.

  \boxed{3} Based on the previous step, now we do not apply integration by parts for $-(\partial_t\vecP,\Delta\vecP)$. Whence,
  \[
    g\|\Delta\vecP\|^2 \le 	c_2\|\GRAD\vecP\|^2 +(\GRAD V_{\vecP},\Delta\vecP) +\rho_v (\partial_t P,\Delta \vecP).
  \]
  We use the Schwarz inequality for the inner-products on the right hand side above and kick back the term $\|\Delta\vecP\|^2$ to the left hand side. This leads to
  \[
    \|\Delta\vecP\|^2 \le C(\|\GRAD\vecP\|^2 + \|\partial_t P\|^2 + \|\rho\|^2) .
  \]
  Invoking $\partial_t\vecP\in L^\infty(0,T;[L^2(\Omega)]^d)$ from Step~1, $\vecP\in L^\infty(0,T;[\Hunz]^d)$ from Step~2, and $\rho \in L^\infty(0,T;L^2(\Omega))$ (due to the continuous embedding from $C^1(0,T;L^2(\Omega))$ to $H^1(0,T;L^2(\Omega))$), there holds
  \[
    \begin{aligned}
      \max_{t\in[0,T]} \|\vecP\|_{[H^2(\Omega)]^d} & \le C\max_{t\in[0,T]} \|\Delta\vecP\|                                                          \\
                                                   & \le C\max_{t\in[0,T]} \bigg(\|\GRAD\vecP\|^2 + \|\partial_t P\|^2  + \|\rho\|^2 \bigg) \le C .
    \end{aligned}
  \]
  The proof is complete.
\end{proof}

\begin{theorem}[uniqueness for the $L^2(\Omega)$ gradient flow]\label{t:uniqueness}
  Under the assumptions provided by Lemma~\ref{l:estimates} and Lemma~\ref{l:max}, solutions in \eqref{e:gf-l2} are unique.
\end{theorem}
\begin{proof}
  Suppose $\vecP_1$ and $\vecP_2$ are two solutions to \eqref{e:gf-l2}. Then their difference satisfies
  \[
    \begin{aligned}
      \rho_v(\partial_t & (\vecP_1-\vecP_2),\vecQ) + g(\GRAD(\vecP_1-\vecP_2),\GRAD\vecQ) \\
                        & + (\GRAD V_{\vecP_1-\vecP_2}, \vecQ)
      = -(\mathbf DF(\vecP_1)-\mathbf DF(\vecP_2),\vecQ) ,\quad\forall \vecQ\in [\Hunz]^d.
    \end{aligned}
  \]
  Since $\vecP_1,\vecP_2$ are bounded in $\Omega\times[0,T]$ due to Lemma~\ref{l:max} and $DF$ is locally Lipschitz, we set $\vecQ = \vecP_1-\vecP_2$ to get
  \[
    \frac{\rho_v}{2}\diff_t\|\vecP_1-\vecP_2\|^2 +g\|\GRAD(\vecP_1-\vecP_2)\|^2
    +\epsilon\|\GRAD(V_{\vecP_1}-V_{\vecP_2})\|^2\le C\|\vecP_1-\vecP_2\|^2,
  \]
  Gr\"{o}nwall's inequality together with the initial condition $\vecP_1(0)-\vecP_2(0)=0$ implies that $\vecP_1=\vecP_2$.
\end{proof}

\section{Numerics}\label{s:numerics}
In this section, we propose a numerical scheme to discretize the $L^2(\Omega)$ gradient flow \eqref{e:gf-l2}. In what follows, we assume that the charge distribution is independent in time, e.g. $\rho = q(N_D - N_A)$ where $q$ denotes the elementary electron charge and $(N_A,N_D)$ are acceptors and donors doping charge densities. % \ref{definition}.

\subsection{A semi-implicit time discretization scheme}
We shall design a stable time discretization scheme with respect to the energy form $I(\vecP)$. For a real number $\lambda$, we denote $\lambda^+=\max\{\lambda,0\}$ and $\lambda^-=\max\{-\lambda,0\}$. The idea is to decompose the term $F(\vecP)$ by a difference of two convex parts: $F(\vecP) = F^+(\vecP) - F^-(\vecP)$, where
\[
  F^\pm(\vecP) = \sum_{i=1}^d \alpha_i^\pm P_i^2 + \beta_i^\pm P_i^4 + \gamma_i^\pm P_i^6 .
\]
We similarly define $\mathbf DF^+(\vecP)$ and $\mathbf DF^-(\vecP)$. Now we set $I(\vecP)=I^+(\vecP)-I^-(\vecP)$ with
\[
  I^+(\vecP) = \int_\Omega \frac{\epsilon}{2}|\GRAD \Vp|^2+ F^+(\vecP) + \sum_{i=1}^d \frac{g_i}2 |\nabla P_i|^2\diff x,
\]
and
\[
  I^-(\vecP) = \int_\Omega F^-(\vecP)  \diff x .
\]
Analogously, we define $G^+(V,\vecP)$ and $G^-(V,\vecP)$ by replacing $\Vp$ with $V$ in $I^\pm(\vecP)$, respectively.

Given a final time $T$ and a positive integer $N$, we set the uniform time step $\tau = \tfrac TN$. The time discretization scheme reads: find a sequence $\{\vecP^n\}_{n=0}^N\subset[\Hunz]^d$ so that $\vecP^0 = \vecP_0$ and for $n\ge 1$,
\begin{equation}\label{e:time-discretization}
  \begin{aligned}
    \rho_v(\diff_{t,\tau} \vecP^{n+1}, & \vecQ)
    + g(\GRAD\vecP^{n+1},\GRAD\vecQ) + (\mathbf DF^+(\vecP^{n+1}),\vecQ)                                                            \\
                                       & + (\GRAD V^{n+1}, \vecQ) =  (\mathbf DF^-(\vecP^n),\vecQ),\quad\forall \vecQ\in[\Hunz]^d ,
  \end{aligned}
\end{equation}
with $V^n = V_{\vecP^n}:=\Phi(\rho^n - \DIV\vecP^n)$ with $\rho^n=\rho(t=n\tau)$ and the backward difference quotient $\diff_{t,\tau}\vecP^n = (\vecP^n - \vecP^{n-1})/\tau$.

\WL{
  \begin{remark}
    The above single-step time stepping scheme is expected to convergence in the first order, namely
    \[
      \|\mathbf P(t=t_n) - \mathbf P^n\| \le C\tau,
    \]
    where the constant $C$ depends on $t_n$. We refer to Section~\ref{s:simulation} for the numerical validation for the above error bound. We also note that multi-step time-stepping schemes, such as BDF-2 or 4th order Runge-Kutta scheme (see the numerical simulations in \cite{lenarczyk2016physical} for \eqref{e:gf-l2}), can be applied to discretize $\partial_t\vecP$. In order to guarantee the higher order convergence in time, we may require that $\vecP$ has higher order time regularity (cf. Theorem~10.1 of \cite{thomee2007galerkin} for the parabolic case). On the other hand, the error analysis for the first-order single-step scheme may require $\vecP\in H^1(0,T;L^2(\Omega))$ but need that $\|\vecP\|_{L^\infty(0,T;H^2(\Omega))}$ is finite (cf. Theorem~8.2 of \cite{thomee2007galerkin}). These low-order regularity conditions are guaranteed by Theorem~\ref{t:existence} and Lemma~\ref{l:max} and allow us to focus on the first order time-stepping scheme in this work. A detailed error analysis as well as the study of the high-order time discretization schemes belongs to our future work.
  \end{remark}
}
\subsection{Energy stability}
We next show that the proposed time discretization scheme is energy stable. To this end, we first rewrite \eqref{e:time-discretization} using the energy form $G(\Vp,\vecP)$. That is, for any $w\in \Hunz_D$ and $Q\in [\Hunz]^d$,
\begin{equation}\label{e:time-discretization-G}
  \left(
  \left(\begin{aligned}
         & 1 &  & 0                 \\
         & 0 &  & \frac{\rho_v}\tau \\
      \end{aligned}\right)
  \left[\left(\begin{aligned} &0\\ &\vecP^{n+1}\end{aligned}\right)
      -\left(\begin{aligned} &0\\ &\vecP^n\end{aligned}\right)\right],
  \left(\begin{aligned} &w\\ &\vecQ\end{aligned}\right)\right)
  = -D G^n\left[\begin{aligned} &w\\ &\vecQ\end{aligned}\right], \quad\text{with } \vecP^0 = \vecP_0,
\end{equation}
where,
\[
  D G^n = \left(\begin{aligned}                                                                                                                                                                                                                                                                                                                                                                                                                                                                                                                                                                                                                                                                                                                                                                                                                         & \delta_V G(V^{n},\vecP^{n})                                               \\
                                                                                                                                                                                                                                                                                                                                                                                                                                                                                                                                                                                                                                                                                                                                                                                                                                        & \delta_{\vecP}G^+(V^{n+1},\vecP^{n+1})-\delta_{\vecP}G^-(V^{n},\vecP^{n})\end{aligned}\right) .
\]
Based on the discussion from Remark~\ref{r:unique}, each term in $G^+$ and $G^-$ is convex with respect to $\vecP$. Whence,
\begin{equation}\label{i:Gc-convex}
  G^+(V^{n+1},\vecP^n) - G^+(V^{n+1}, \vecP^{n+1}) \le \delta_\vecP G^+(V^{n+1}, \vecP^{n+1})[\vecP^n - \vecP^{n+1}]
\end{equation}
and
\begin{equation}\label{i:Ge-convex}
  G^-(V^n,\vecP^{n+1}) - G^-(V^n, \vecP^{n})
  \ge \delta_{\vecP}G^-(V^n, \vecP^{n})[\vecP^{n+1} - \vecP^n] .
\end{equation}

\begin{proposition}\label{p:stable}
  The time discretization scheme \eqref{e:time-discretization} is energy stable, \ie
  \[
    I(\vecP^{n+1}) \le I(\vecP^n) .
  \]
\end{proposition}
\begin{proof}
  According to \eqref{i:Gc-convex} and adding and subtracting the term $G^+(V^{n},\vecP^{n})$, we obtain that
  \begin{equation}\label{e:diff-1}
    \begin{aligned}
      \delta_{\vecP}G^+(V^{n+1}, \vecP^{n+1})[\vecP^{n+1} - \vecP^n] & \ge G^+(V^{n+1}, \vecP^{n+1}) - G^+(V^{n+1},\vecP^n) \\
                                                                     & = G^+(V^{n+1},\vecP^{n+1}) - G^+(V^n,\vecP^n) + D .
    \end{aligned}
  \end{equation}
  Here
  \begin{equation}\label{e:D}
    \begin{aligned}
      D & = G^+(V^n,\vecP^{n}) - G^+(V^{n+1},\vecP^{n})                                                                  \\
        & = \int_\Omega -\frac\epsilon2 |\GRAD V^n|^2 +\rho V^n + \vecP^n\cdot \GRAD V^n \diff x                         \\
        & \quad\quad- \int_\Omega -\frac\epsilon2|\GRAD V^{n+1}|^2 + \rho V^{n+1}  + \vecP^n\cdot \GRAD V^{n+1}\diff x . \\
    \end{aligned}
  \end{equation}
  In view of \eqref{e:poisson}, there holds
  \[
    \int_\Omega \rho(V^n-V^{n+1}) +\vecP^n\cdot\GRAD(V^n- V^{n+1})\diff x
    = \epsilon\int_\Omega \GRAD V^n\cdot\GRAD(V^n- V^{n+1}) \diff x.
  \]
  Inserting the above formula into \eqref{e:D} yields $D\ge 0$. Hence,
  \begin{equation}\label{i:st-b-1}
    \delta_\vecP G^+(V^n, \vecP^{n+1})[\vecP^{n+1} - \vecP^n]\ge G^+(V^{n+1},\vecP^{n+1}) - G^+(V^n,\vecP^n).
  \end{equation}
  Following \eqref{i:Ge-convex} and noting that $G^-$ is independent of $V$, one can obtain that
  \begin{equation}\label{i:st-b-2}
    \begin{aligned}
      G^-(V^{n+1},\vecP^{n+1}) - G^-(V^n, \vecP^{n}) & = G^-(V^n,\vecP^{n+1}) - G^-(V^n, \vecP^{n})                   \\
                                                     & \ge \delta_{\vecP}G^-(V^n, \vecP^{n})[\vecP^{n+1} - \vecP^n] .
    \end{aligned}
  \end{equation}
  Letting $(w,\vecQ)^{\Tr} = D G^n$ in \eqref{e:time-discretization-G} and invoking \eqref{i:st-b-1} and \eqref{i:st-b-2} we conclude that
  \[
    \begin{aligned}
      0 & \ge -(DG^n,DG^n)                                                                                                   \\
        & = \frac{\rho_v}{\tau}(\delta_\vecP G^+(V^n,\vecP^{n+1})[\vecP^{n+1}-\vecP^n]
      - \delta_\vecP G^-(V^{n},\vecP^{n})[\vecP^{n+1}-\vecP^n])                                                              \\
        & \ge\frac{\rho_v}{\tau}[G^+(V^{n+1},\vecP^{n+1}) - G^+(V^n,\vecP^n)  - (G^-(V^{n+1},\vecP^{n+1})-G^-(V^n,\vecP^n))] \\
        & =  \frac{\rho_v}{\tau}[G(V^{n+1},\vecP^{n+1}) - G(V^n,\vecP^n)].
    \end{aligned}
  \]
  The proof is complete according to Proposition~\ref{p:gf-equivalent}.
\end{proof}

\subsection{An HDG space discretization}\label{s:space}
Based on the time-discretized problem \eqref{e:time-discretization}, we next proposed an HDG space discretization scheme inspired by \cite{cockburn2009unified}.

\WL{Let $\{\mathcal T_h\}_{h>0}$ be a family of conforming quasi-uniform subdivisions of $\Omega$ made of shape-regular simplices or regular quadrilaterals (cf. \cite{ciarlet2002finite})} in a sense that there exists a uniform constant $c_q>0$ in dependent of $\mathcal T_h$ such that
\[
  h:=\max_{K\in\mathcal T_h} h_K = c_q \min_{K\in\mathcal T_h} h_K ,
\]
where $h_K$ denotes the size of a cell $K\in\mathcal T_h$. Denote $\mathcal E_h$ the skeleton of $\mathcal T_h$, \ie the collection of all faces of $K\in \mathcal T_h$. The space discretization is based on the mixed formulation of \eqref{e:time-discretization} in each element $K\in\mathcal T_h$. For the Poisson problem \eqref{e:poisson} at $t_n = n\tau$, knowing that the electric field $\vecE = -\GRAD V$, we write
\begin{equation}\label{e:mixed-poisson}
  \begin{aligned}
    (\vecE^n, \vecr)_K - (V^n,\DIV \vecr)_K + (\widehat V^n, \vecr \cdot \nu)_{\partial K}                           & = 0,              \\
    -(\epsilon \vecE^n +\vecP^n,\GRAD w)_K + ((\epsilon \widehat\vecE^n + \widehat\vecP^n)\cdot \nu, w)_{\partial K} & = (\rho^n, w)_K ,
  \end{aligned}
\end{equation}
where $\widehat\cdot$ denotes the trace on $\partial K$ and $\nu$ is the outward normal vector of $K$. For equation \eqref{e:time-discretization}, setting $\mathbb W=-g\GRAD\vecP$, the local problem becomes
\begin{equation}\label{e:mixed-polarization}
  \begin{aligned}
    (g^{-1} \mathbb U^n,\mathbb W)_K - (\vecP^n,\DIV\mathbb W)_K + (\widehat \vecP^n, \mathbb W\cdot \nu)_{\partial K}                    & = 0,                                   \\
    \rho_v(\diff_{\tau,t}\vecP^n,\vecQ)_K + (\mathbf DF^+(\vecP^n)                                                     - \vecE^n,\vecQ)_K & - (\mathbb U^n, \GRAD \vecQ)_K         \\
    + (\widehat{\mathbb U}^n\cdot \nu,\vecQ)_{\partial K}                                                                                 & = (\mathbf DF^-(\vecP^{n-1}),\vecQ)_K.
  \end{aligned}
\end{equation}
We also need to \WL{weakly} impose the so-called transmission conditions to glue the local problems for adjacent elements. That is for each interior face $F$,
\begin{equation}\label{e:transmission}
  ([ (\epsilon\widehat\vecE^n +\widehat\vecP^n)\cdot \nu],\xi)_F = ([ \epsilon\widehat\vecE^n\cdot \nu],\xi)_F = 0,
  \quad \text{ and } \quad
  ([ \widehat{\mathbb U}^n\cdot \nu], \bm\eta) = 0 .
\end{equation}
Here $[\cdot]$ denotes the jump of a function (or vector field) at $F$ and $\xi,\bm\eta$ are the corresponding test function and vector field. We note that for the first equation in \eqref{e:transmission} used the fact that the $[\vecP] =0$ on $F$ according to Theorem~\ref{t:existence}.

Now we define the numerical scheme for the triplets $(V,\vecE, \widehat V)$ and $(\vecP,\mathbb U, \widehat{\vecP})$. \WL{Let $\vk{k}$ be the space of piecewise polynomials for simplicial meshes or bi-polynomials for quadrilateral meshes of degree at most $k$ subordinate to $\mathcal T_h$ and define $\hvk{k}$ to be the space of piecewise polynomials or bi-polynomials subordinate to $\mathcal E_h$.} We also denote $\hvk{k}(0)$ and $\hvkd{k}(0)$ as the collections of functions in $\hvk k$ vanishing on $\partial\Omega$ and $\Gamma_D$, respectively. Letting
\[
  (\cdot,\cdot)_{\mathcal T_h} = \sum_{K\in\mathcal T_h} (\cdot,\cdot)_K
  \quad\text{ and }\quad
  (\cdot,\cdot)_{\mathcal E_h} = \sum_{K\in \mathcal T_h} (\cdot,\cdot)_{\partial K},
\]
The \ac{HDG} discretization of \eqref{e:mixed-poisson} and \eqref{e:mixed-polarization} reads: find a sequence of functions $(V^n_h, \vecE^n_h, \widehat{V}^n_h)\in \vk k\times [\vk k]^d\times \hvkd k(0)$ and $(\vecP^n_h,\mathbb U^n_h,\widehat\vecP^n_h)\in [\vk k]^d\times[\vk k]^{d\times d}\times [\hvk k(0)]^d$ such that for $n=0$, $(\vecP^0_h,\vecQ)=(\vecP_0,\vecQ)$ holds for all $\vecQ\in [\vk k]^d$ and for $n\ge 1$,
\begin{equation}\label{e:full-system}
  \begin{aligned}
    (\vecE^n_h, \vecr)_{\mathcal T_h} - (V^n_h,\DIV \vecr)_{\mathcal T_h} + (\widehat V^n_h, \vecr \cdot \nu)_{\mathcal E_h}                             & = 0,                                                   \\
    -(\epsilon \vecE^n_h +\vecP^n_h,\GRAD w)_{\mathcal T_h} + ((\epsilon \widehat\vecE^n_h + \widehat\vecP^n_h)\cdot \nu, w)_{\mathcal E_h}              & = (\rho^n, w)_{\mathcal T_h} ,                         \\
    (g^{-1} \mathbb U^n_h,\mathbb W)_{\mathcal T_h} - (\vecP^n_h,\DIV\mathbb W)_{\mathcal T_h} + (\widehat \vecP^n_h, \mathbb W\cdot \nu)_{\mathcal E_h} & = 0,                                                   \\
    \rho_v(\diff_{\tau,t}\vecP^n_h,\vecQ)_K + (\mathbf DF^+(\vecP^n_h) - \vecE^n_h,\vecQ)_K                                                              & - (\mathbb U^n_h, \GRAD \vecQ)_K                       \\
    + (\widehat{\mathbb U}^n_h\cdot \nu,\vecQ)_{\mathcal E_h}                                                                                            & = (\mathbf DF^-(\vecP^{n-1}_h),\vecQ)_{\mathcal T_h} ,
  \end{aligned}
\end{equation}
hold for all $(\vecr,w)\in [\vk k]^d\times\vk k$ and $(\mathbb W, \vecQ)\in [\vk k]^{d\times d}\times[\vk k]^d$. Here the numerical traces for fluxes $\widehat\vecE$ and $\widehat{\mathbb U}$ are defined by
\begin{equation}\label{e:numerical-fluxes}
  \begin{aligned}
    (\epsilon\widehat \vecE^n_h+\widehat\vecP^n_h)\cdot \nu
                                     & = (\epsilon\vecE^n_h+\widehat\vecP^n_h)\cdot\nu +\tau_{V}(V^n_h-\widehat V^n_h), \\
    \widehat{\mathbb U}^n_h\cdot \nu & = \mathbb U^n_h\cdot\nu
    +\tau_{\vecP}(\vecP^n_h-\widehat\vecP^n_h) ,
  \end{aligned}
\end{equation}
where $\tau_V$ and $\tau_\vecP$ are \WL{positive} piecewise constant functions defined on $\mathcal E_h$ with order $O(1)$ with respect to $h$ (cf. \cite{cockburn2009unified}). To guarantee the scaling of the equations, we set $\tau_V = \epsilon/\text{diam}(\Omega)$ and $\tau_\vecP = 1/\text{diam}(\Omega)$. Letting
\[
  (\cdot,\cdot)_{\mathcal E_h\backslash\partial\Omega} = \sum_{K\in \mathcal T_h} (\cdot,\cdot)_{\partial K\backslash\partial\Omega} ,
\]
we shall also write the discrete counterpart of the transmission conditions \eqref{e:transmission} together with the boundary conditions as the following: for all $\xi\in \hvkd k(0)$ and $\bm\eta\in[\hvk k(0)]^d$, there hold
\begin{equation}\label{e:transmission-discrete}
  \begin{aligned}
    (\epsilon\widehat{\vecE}^n_h\cdot\nu,\xi)_{\mathcal E_h\backslash\partial\Omega}  & = 0 , \\
    (\widehat{\mathbb U}^n_h\cdot \nu,\bm\eta)_{\mathcal E_h\backslash\partial\Omega} & = 0 .
  \end{aligned}
\end{equation}
and
\begin{equation}\label{e:boundary-discrete}
  \WL{
    \begin{aligned}
      \widehat{\vecP}_h^n                                                      & = 0, \quad \text{ on }\partial\Omega    ,                     \\
      \widehat{V}_h^n                                                          & = 0,\quad \text{ on }\Gamma_D    ,                            \\
      ((\epsilon\widehat{\vecE}^n_h+\widehat\vecP^n_h)\cdot\nu,\xi)_{\Gamma_N} & = (\epsilon\widehat{\vecE}^n_h\cdot\nu,\xi)_{\Gamma_N}  = 0 .
    \end{aligned}}
\end{equation}
Equations \eqref{e:full-system}--\eqref{e:boundary-discrete} form the fully discrete system and the well-posedness of this system is guaranteed by Lemma~3.1, Proposition~3.3 and Theorem~2.4 from \cite{cockburn2009unified}. \WL{Assuming that the primal variables $\vecP$, $V$, $\vecE$ and $\mathbb U$  are smooth in space, the $L^2(\Omega)$ errors for the fully discrete approximation are expected to get bounded by (cf. \cite{cockburn2010projection})
  \begin{equation}\label{i:expect-err}
    \|\vecP(t=t_n) - \vecP^n_h\|_{L^2(\Omega)}
    +\|V(t=t_n) - V^n_h\|_{L^2(\Omega)}
    +\|\vecE(t=t_n) - \vecE^n_h\|_{L^2(\Omega)}+
    \|\mathbb U(t=t_n) - \mathbb U^n_h\|_{L^2(\Omega)}\le C(h^{k+1} +\tau) .
  \end{equation}
}

\section{Numerical illustration}\label{s:simulation}
In this section, we provided some numerical tests to verify the stability and convergence of the fully discrete scheme \eqref{e:full-system}--\eqref{e:transmission-discrete}. All numerical tests are implemented using the \texttt{deal.II} finite element library \cite{arndt2020deal} using quadrilateral meshes in the two dimensional space. We shall use piecewise bilinear or bi-quadratic elements in our numerical tests. We finally note out that since \eqref{e:full-system} is a nonlinear system, we solve it using the Newton iteration and each iteration is solved by a direct solver based on the \WL{``klu'' method built in the \WL{\texttt{Aemesos}} package} \footnote{https://docs.trilinos.org/dev/packages/amesos/doc/html/index.html}.

\subsection{Convergence tests}
Let $\Omega=(0,1)^2$. We apply the fully-discrete scheme \eqref{e:full-system} to discretize following the time dependent problem
\begin{equation}\label{e:system-modified}
  \begin{aligned}
    \DIV(-\GRAD V + \vecP)                                    & = 0,\quad             &  & \text{ in } \Omega\times(0,T],      \\
    \partial_t\vecP -\Delta\vecP - \mathbf DF(\vecP) +\GRAD V & = \mathbf G,\quad     &  & \text{ in } \Omega\times(0,T],      \\
    \vecP                                                     & = 0,\quad             &  & \text{ on }\partial\Omega,          \\
    V                                                         & = 0,\quad             &  & \text{ on } \{x_2=0\}\cup\{x_2=1\}, \\
    (-\GRAD V + \vecP)\cdot \nu                               & =0,\quad              &  & \text{ on } \{x_1=0\}\cup\{x_1=1\}, \\
    \vecP(0)                                                  & =\vecP_0,       \quad &  & \text{ in }\Omega .
  \end{aligned}
\end{equation}
In order to check the convergence of the proposed scheme, we set the analytic solutions
\[
  \begin{aligned}
    V   & = e^{-2\pi^2 t}\sin(\pi x_1) \sin(\pi x_2),                       \\
    P_1 & = -E_1 = e^{-2\pi^2 t}\pi\cos(\pi x_1)\sin(\pi x_2), \text{ and } \\
    P_2 & = -E_2 = e^{-2\pi^2 t}\pi\sin(\pi x_1)\cos(\pi x_2)
  \end{aligned}
\]
to obtain the data function $\mathbf G$ and and initial polarization field $\vecP_0$. We also note that when considering the approximation of $\vecP^n$, we should evaluate $\mathbf G$ at $t_n = \tau n$.

\subsubsection*{Convergence in time} To test the convergence in time, we use a fixed uniform mesh together with continuous piecewise bilinear elements. The mesh size is small enough so that the error from the time discretization dominates the total error. Under a uniform subdivision with $36864$ \WL{degrees of freedom}, Figure~\ref{f:time-convergence} reports $L^2(\Omega)$ errors of $\vecE$, $V$, $\mathbb U$ and $\vecP$ at the final time $T=0.1$ against the time step $\tau$. We observe that all error plots decay in the first order. We also point out that the error plots for $\vecE$ and $\vecP$ almost coincide in Figure~\ref{f:time-convergence}.
\begin{figure}[hbt!]
  \begin{center}
    \includegraphics[scale=0.5]{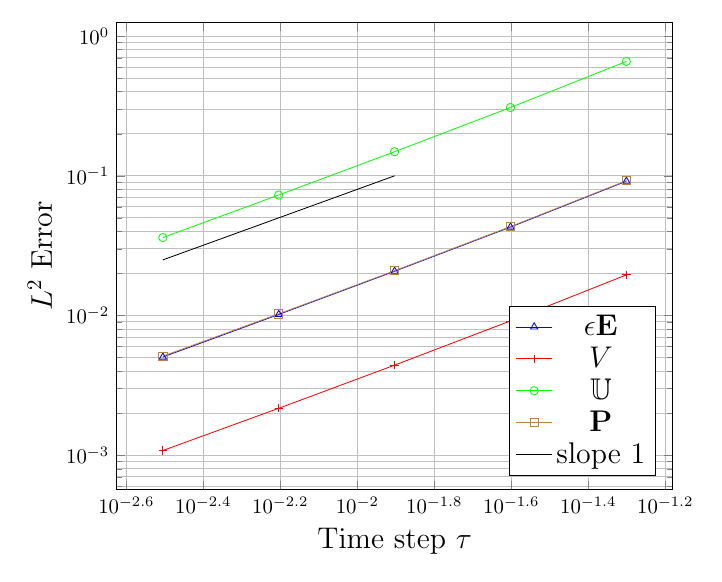}
  \end{center}
  \caption{$L^2(\Omega)$ errors of each primal variables in the system \eqref{e:system-modified} against the time step $\tau$. Note that the errors for $\vecE$ and $\vecP$ are so closed that their plots almost coincide.}
  \label{f:time-convergence}
\end{figure}

\subsubsection*{Convergence in space}
\WL{We next check the rate of convergence in space with an indirect approach. The idea is to guarantee that the error from the time discretization does not dominate the total error. To this end, we start to solve the discrete system under a uniform coarse mesh containing $16$ cells and setting the time step $\tau=T$. Then we refine the mesh globally and divide the time step by $2^{k+1}$, where we recall that $k=1$ or $2$ denotes the degree of bi-polynomials. So the time step $\tau = ch^{k+1}$ and according to the expected error bound \eqref{i:expect-err}, the total error should behave like $O(h^{k+1})$.}

Figure~\ref{f:space-convergence} reports the $L^2(\Omega)$-error against the number of \WL{degrees of freedom} (\#DoFs) at $T=0.1$ using both bilinear and bi-quadratic elements. We observe that all $L^2(\Omega)$ errors for primal variables behave like $O(h^{k+1})$ since $h\sim \#\text{DoFs}^{-1/d}$.

\begin{figure}[hbt!]
  \begin{center}
    \includegraphics[scale=0.5]{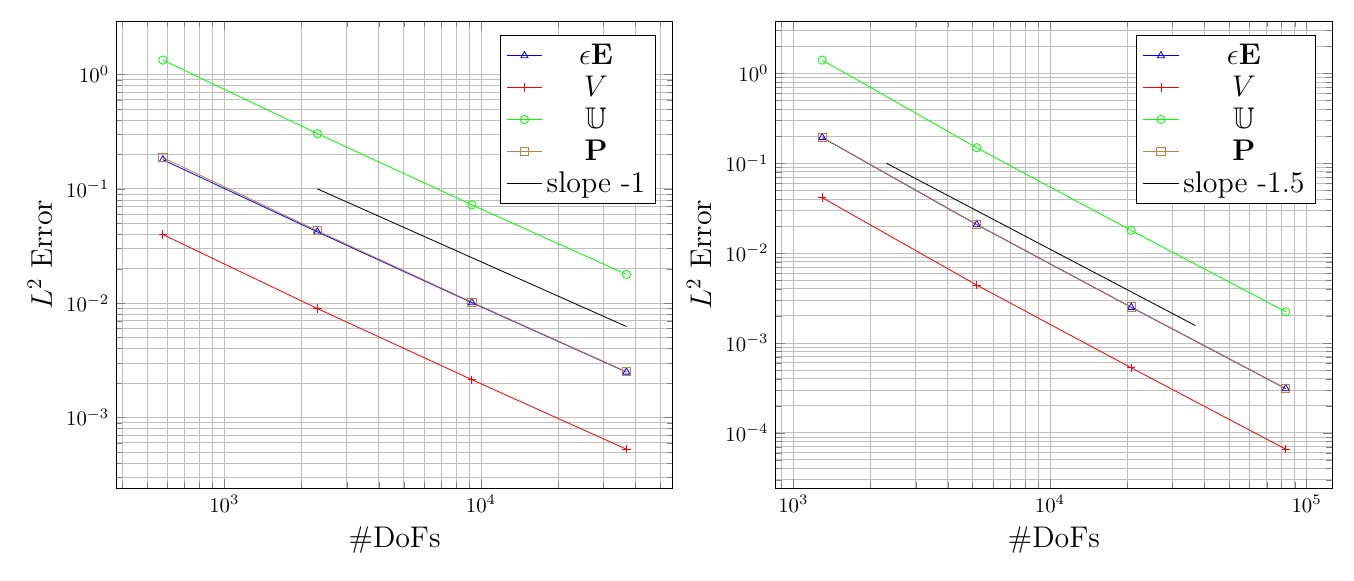}
  \end{center}
  \caption{$L^2(\Omega)$ errors for each primal variable in the system \eqref{e:system-modified} against \#DoFs using continuous piecewise bilinear (left) and piecewise bi-quadratic elements (right), respectively. Note that similar to Figure~\ref{f:time-convergence}, the errors for $\vecE$ and $\vecP$ are so closed that their plots almost coincide.}
  \label{f:space-convergence}
\end{figure}

\subsection{Monolayer tests}
Now we test the numerical scheme \eqref{e:full-system} in a monolayer device which is $80n m$ long and $40 n m$ wide. Figure~\ref{f:mesh} shows a uniform subdivision of the test device. In terms of the time discretization, we set the final time $T=160ns$ and the time step $\tau=T/1000$. For the parameters of the \ac{GLD} model \eqref{e:energy-component-wise}, we set
$\epsilon_b=5F/ m$, $\alpha=-1.54\times 10^9 m/ F$, $\beta=-2.65\times 10^{12}m^5/ FC^2$, $\gamma = 2.6\times 10^{15} m^9/ FC^4$ and $g=10^{-8} m^3/ F$.

\begin{figure}[hbt!]
  \begin{center}
    \includegraphics[scale=0.15]{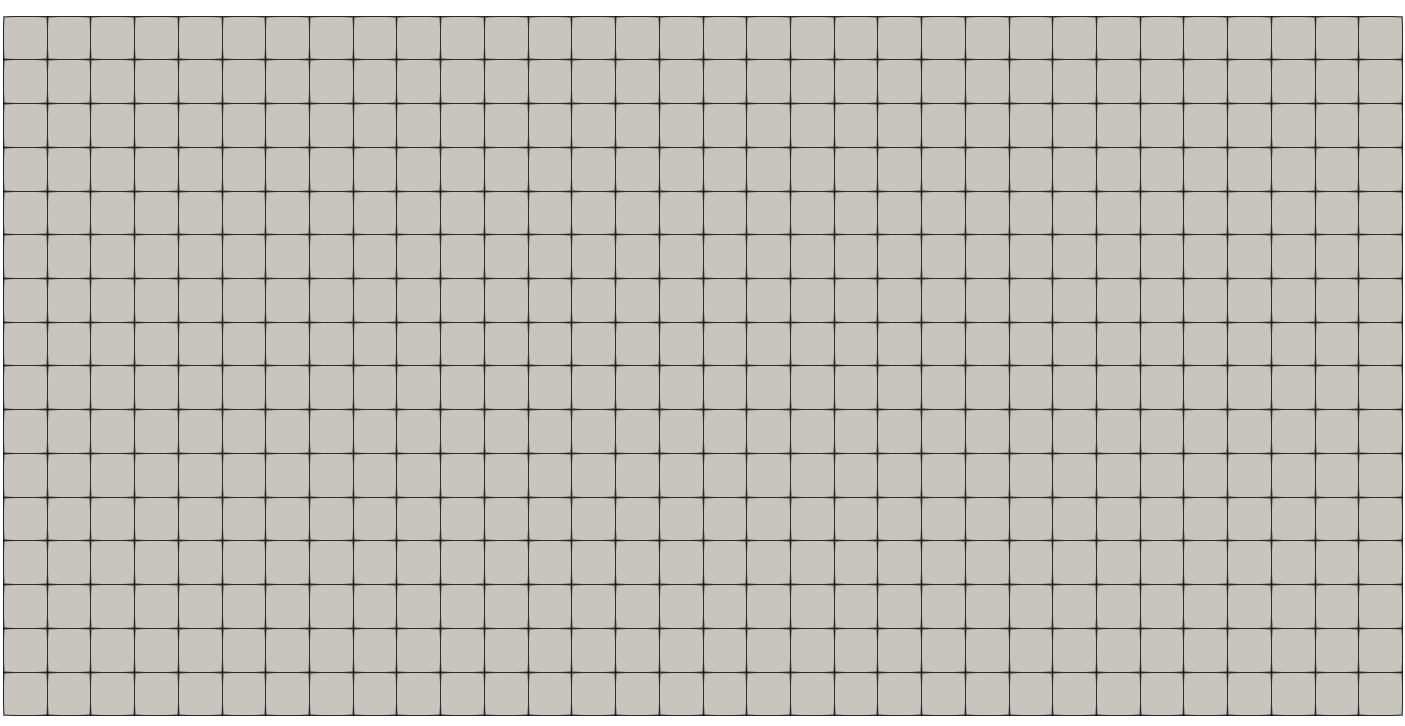}
  \end{center}
  \caption{Subdivision of the monolayer test device.}
  \label{f:mesh}
\end{figure}

\subsubsection{Test for the energy stability}
We first numerically check the energy stability of our proposed scheme. Here we set zero electric distribution in the device (\ie $\rho=0$) and apply zero voltages at both the top and bottom. We also set the zero Neumann boundary conditions on both left and right boundaries of the device. That is $(\epsilon\vecE+\vecP)\cdot\nu = 0$ on $\{x_1=0\}\cup\{x_1 = 80n m\}$. Starting from the initial polarization field
\[
  P_i(t=0) = \left\{
  \begin{aligned}
    0.1 V/m,  & \quad \text{ if } x_1\le 40 nm, \\
    -0.1 V/m, & \quad \text{ if } x_1\ge 40 nm, \\
  \end{aligned}
  \right .
\]
we plot the approximation of the density of the total energy \eqref{e:energy}, \ie
\[
  d_h(\vecP_h^n) := \frac{1}{|\Omega|}\int_{\Omega}\frac{\epsilon}{2}|\GRAD V_h^n|^2
  + F(\vecP_h^n) + \sum_{i=1}^2 \frac{1}{2g}|\mathbb U_h^n|^2 \diff x,
\]
in the time range $[0,160ns]$ in Figure~\ref{f:energy}. In order to observe the decay of the energy density, we also plot $\text{sign}(d_h)\log_{10}(d_h)$ \WL{versus} time in Figure~\ref{f:energy} (noting that the minimum energy is negative). The approximation of the polarization field $\vecP_h$ at $80ns$ and $160ns$ as well as corresponding $\ROT\vecP_h$ are provided in Figure~\ref{f:polarization-nonzero}.

\begin{figure}[hbt!]
  \begin{center}
    \begin{tabular}{cc}
      \includegraphics[width=0.5\textwidth]{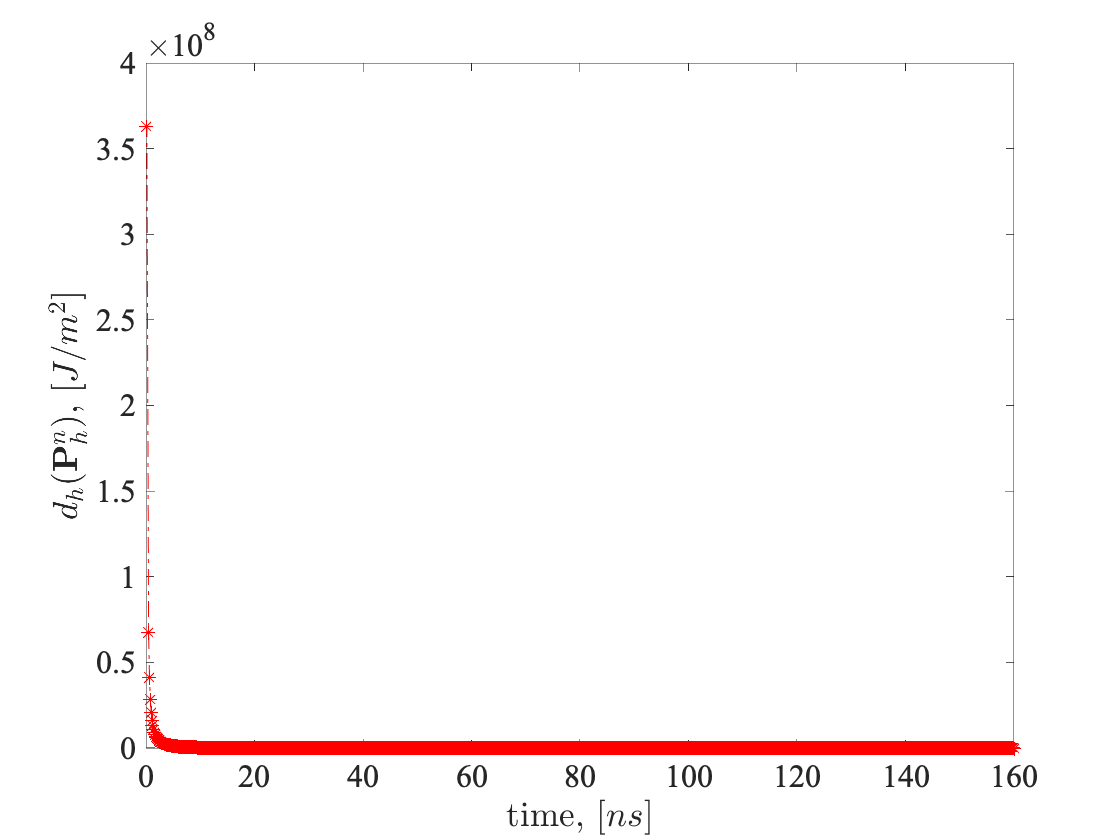} & \includegraphics[width=0.5\textwidth]{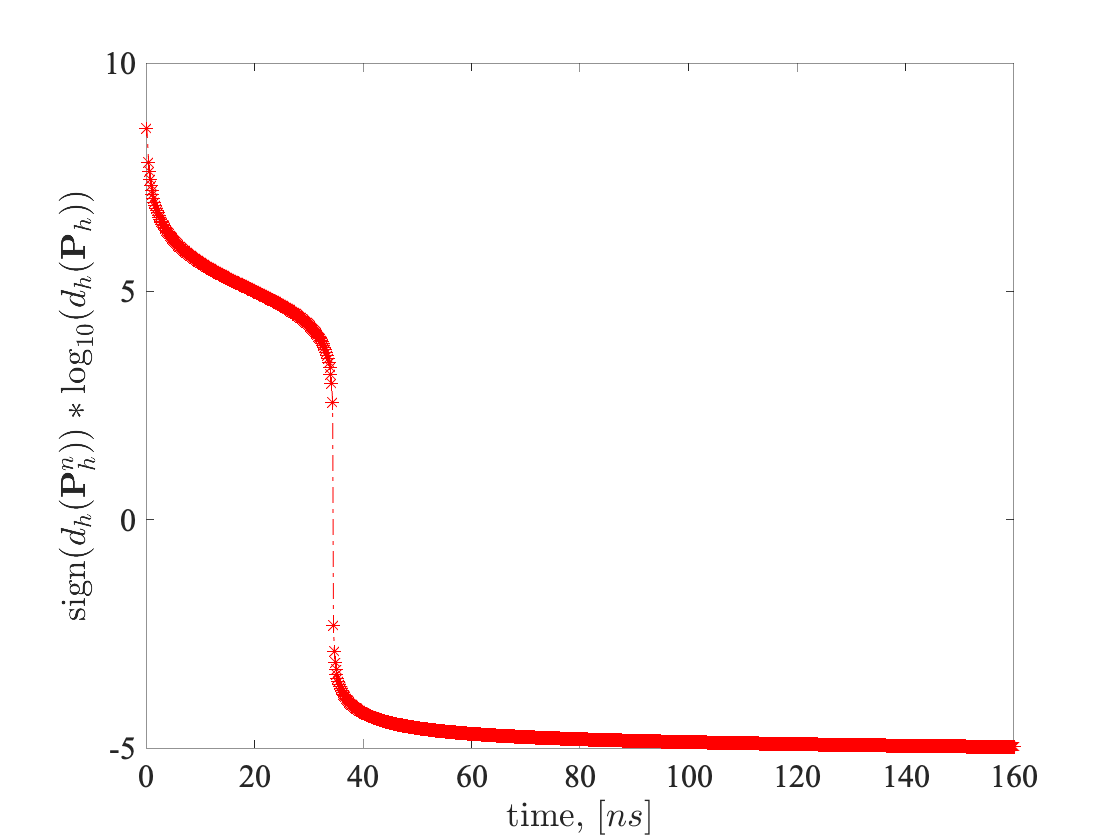} \\
    \end{tabular}
  \end{center}
  \caption{Test for the energy stability: decay of the energy (left) from $0ns$ and $160ns$ and its logarithm plot (right).}
  \label{f:energy}
\end{figure}

\begin{figure}[hbt!]
  \begin{center}
    \begin{tabular}{cc}
      \includegraphics[width=0.5\textwidth]{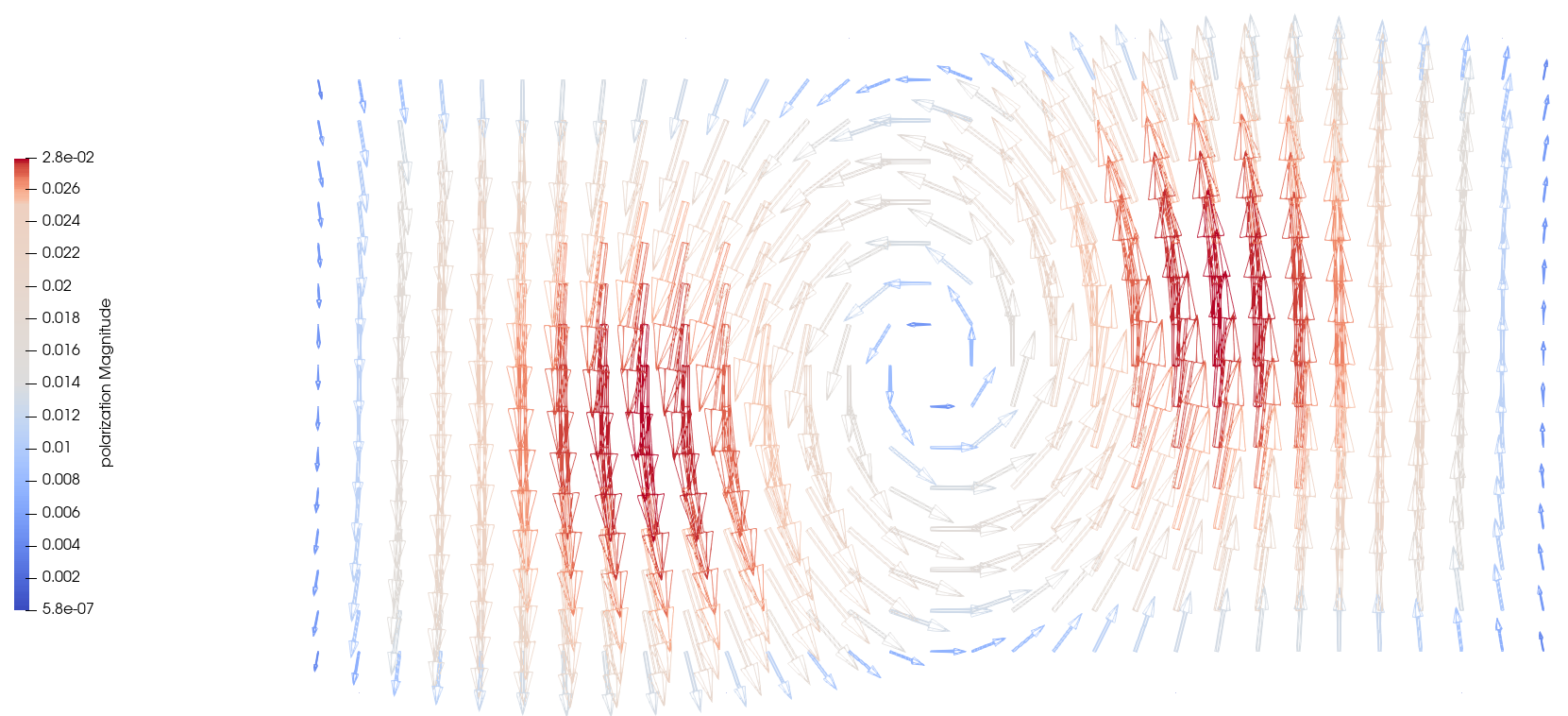}     & \includegraphics[width=0.5\textwidth]{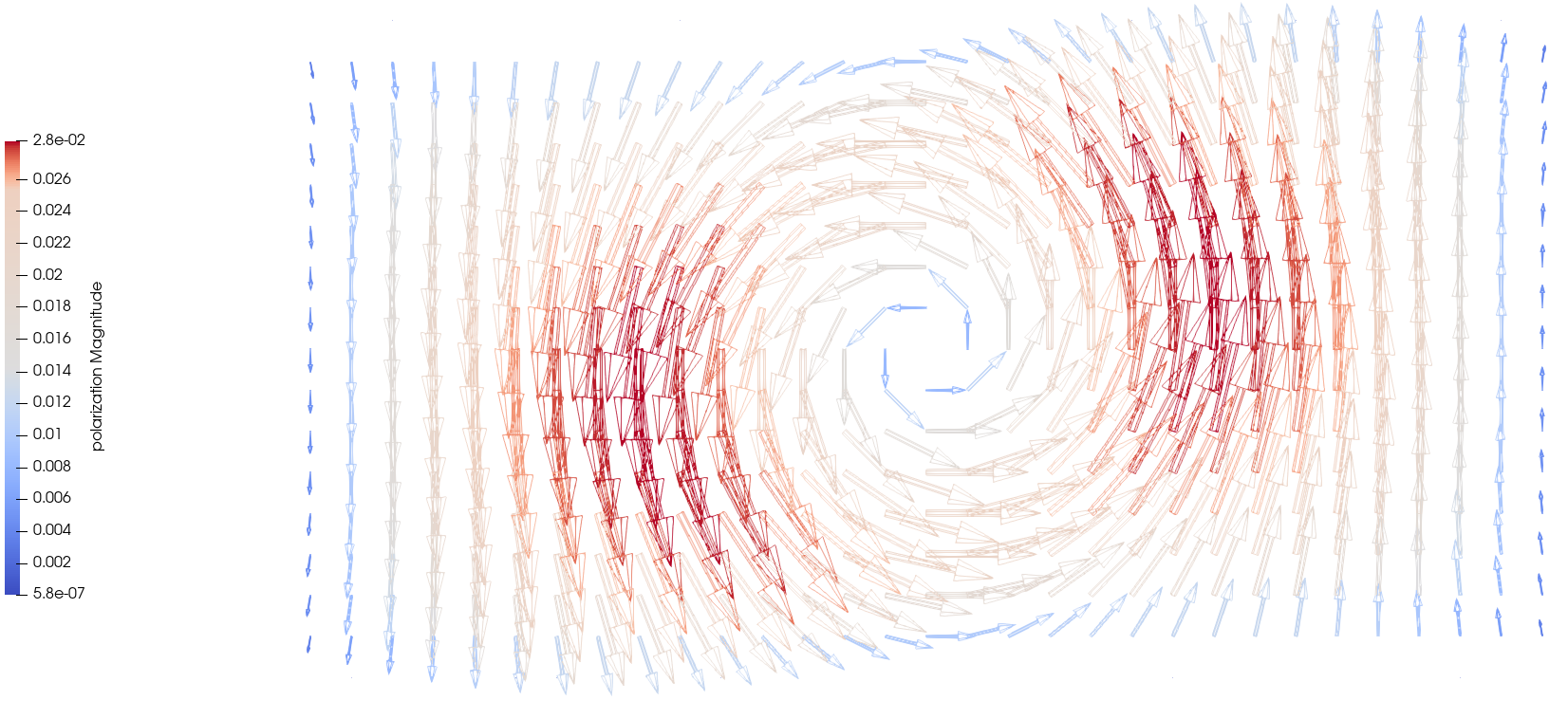}     \\
      \includegraphics[width=0.5\textwidth]{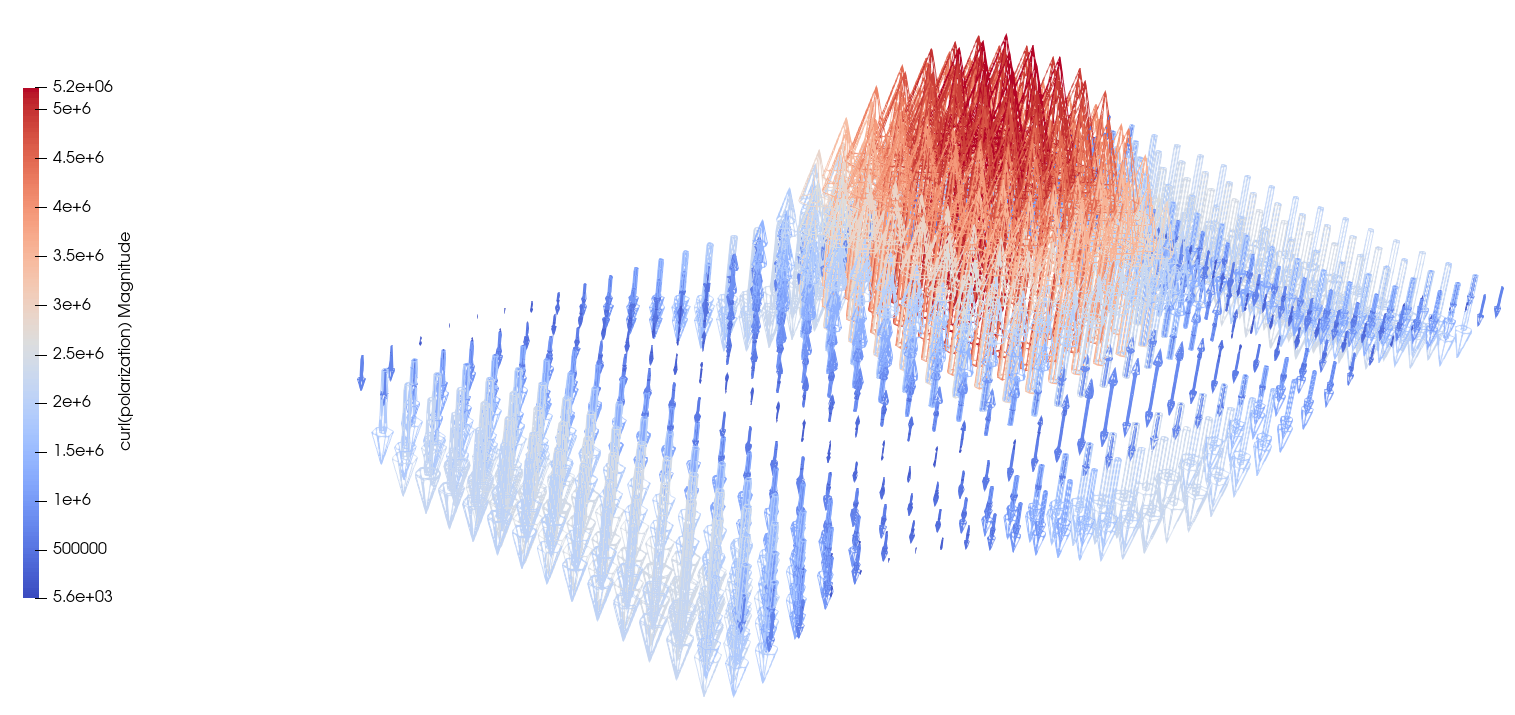} & \includegraphics[width=0.5\textwidth]{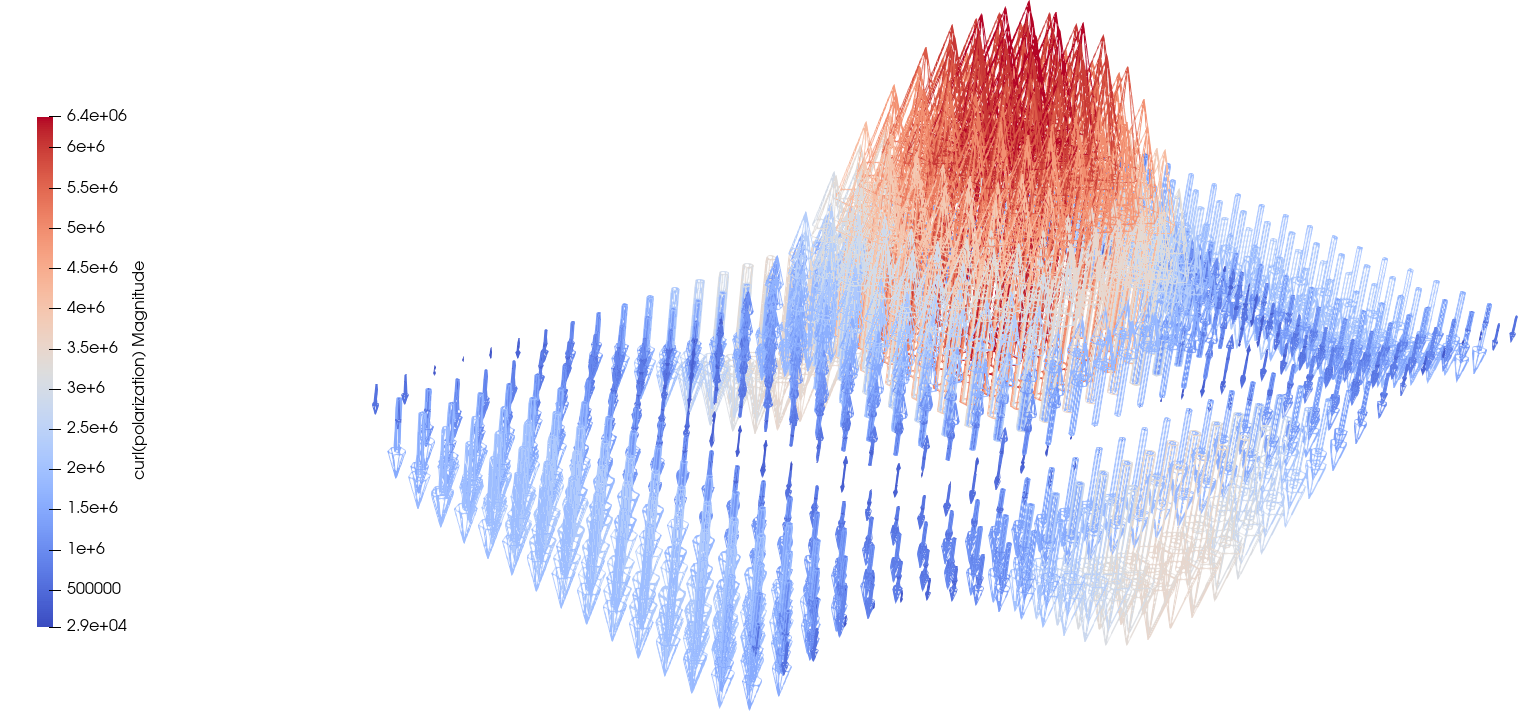} \\
    \end{tabular}
  \end{center}
  \caption{Test for the energy stability: approximations of the polarization field $\vecP_h$ at $80ns$ (top-right) and $160ns$ (top-right) as well as the plots for $\ROT\vecP_h$ at $80ns$ (bottom-right) and $160ns$ (bottom-right) . The magnitudes of $\vecP_h$ and $\ROT\vecP_h$ are rescaled by factors $3\times 10^{-7}$ and $4\times 10^{-15}$ respectively.}
  \label{f:polarization-nonzero}
\end{figure}

\subsubsection{An integration test}
Starting from the same discretization settings from the previous section, we fix the zero voltage at the bottom of the test device and apply different voltages on the top side with a contact bias following a triangle signal whose the first period is defined as
\[
  V(t) = \left\{
  \begin{aligned}
    5t\, V,      & \quad \text{ if } 0\le t\le 20ns,     \\
    200-5t\, V,  & \quad \text{ if } 20ns \le t\le 60ns, \\
    -400+5t\, V, & \quad \text{ if } 60ns \le t\le 80ns; \\
  \end{aligned}
  \right .
\]
see also the signal plot in Figure~\ref{f:signal}. In order to obtain accurate approximations, we locally refine the subdivision provided by Figure~\ref{f:mesh} every five time steps. More precisely speaking, after the computation at the time step $n=5k$ with $k=0,1,2,\ldots$, we estimate the error of the solution $V_h^n$ on each cell $K$ using the so-called Kelly error estimator \cite{kelly1983posteriori}:
\[
  \eta_K^2 = \sum_{F\in\partial K} \frac{\text{diam}(K)}{24}[\GRAD V^n_h\cdot \nu]^2 .
\]
Then, we refine cells whose corresponding $\eta_K$ are 1\% largest using the quad-refinement strategy (cf. \cite{bonito2010quasi}). The left plot of Figure~\ref{f:displacement} reports the discrete counterpart of the displacement vector
\[
  \int_S \vecD(t_n)\cdot\nu \diff s \approx \mathscr D := \int_S (\epsilon \widehat{\vecE}^n_h +\widehat\vecP^n_h)\cdot\nu \diff s,
\]
at top ($S=\{x_2=40ns\}$) and bottom ($S=\{x_2=0ns\}$) of the device against contact bias in one and a half periods (120ns in total). The right of Figure~\ref{f:displacement} shows the corresponding displacement current density $\int\partial_t \vecD\cdot \nu \diff s$ \WL{versus} time. Here we approximate this quantity by using the difference quotient $\diff_{t,\tau}\mathscr D$. Both plots in Figure~\ref{f:displacement} shows that the polarization field switches its direction along $x_2$ direction. In order to show the dynamic of the polarization field, we report in Figure~\ref{f:polarization-zero} the $x_2$-component of $\widehat{\vecP}^n_h$ along $x_1=40ns$ at $t=20j\, ns$ for $j=1,\ldots,5$. The corresponding background subdivisions are reported in Figure~\ref{f:polarization-zero} as well. To illustrate the improvement of the accuracy from the adaptive refinement strategy, we plot $\widehat V^n_h$ along $x_1 = 40ns$ both on the meshes in Figure~\ref{f:polarization-nonzero} and on the coarse mesh in Figure~\ref{f:mesh}.

\begin{figure}[hbt!]
  \begin{center}
    \includegraphics[scale=0.35]{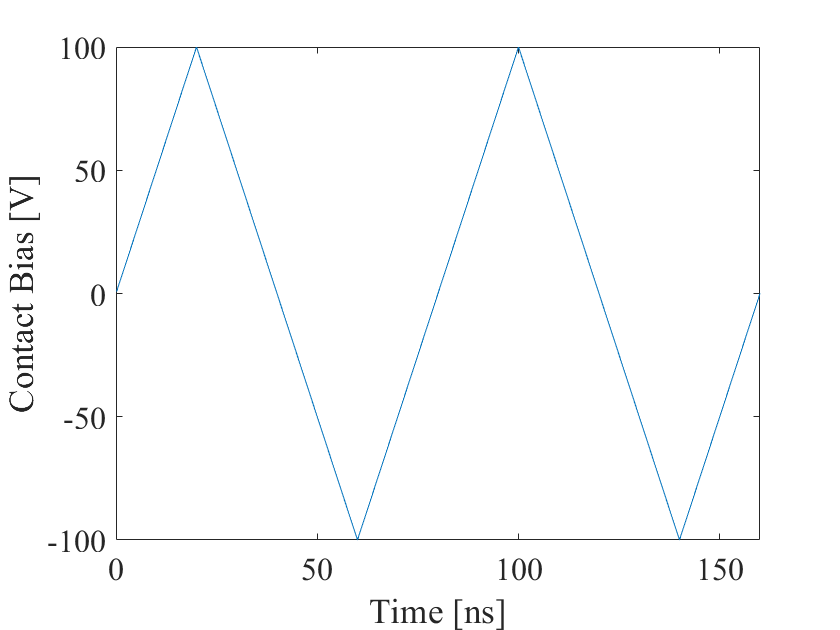}
  \end{center}
  \caption{Signal of the contact bias in two periods.}
  \label{f:signal}
\end{figure}

\begin{figure}[hbt!]
  \begin{center}
    \begin{tabular}{cc}
      \includegraphics[width=0.5\textwidth]{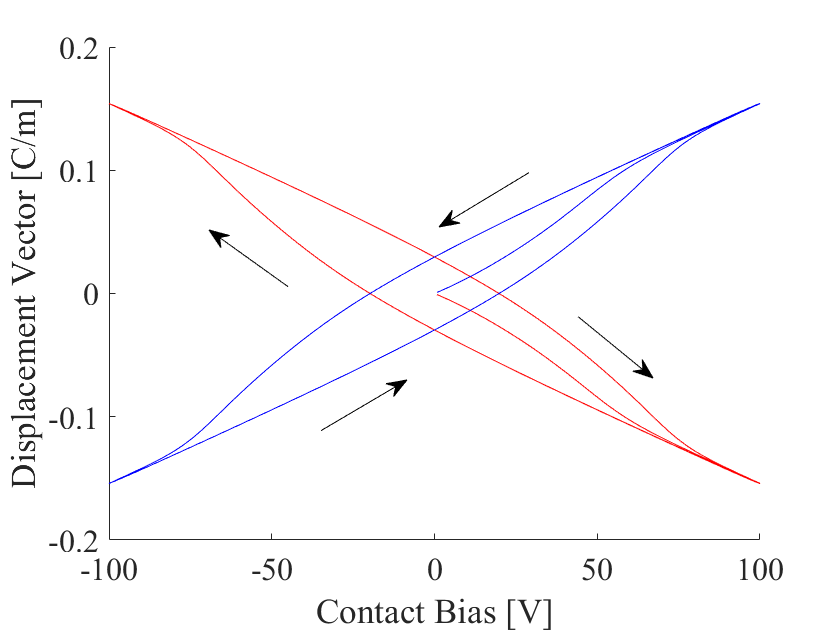} & \includegraphics[width=0.5\textwidth]{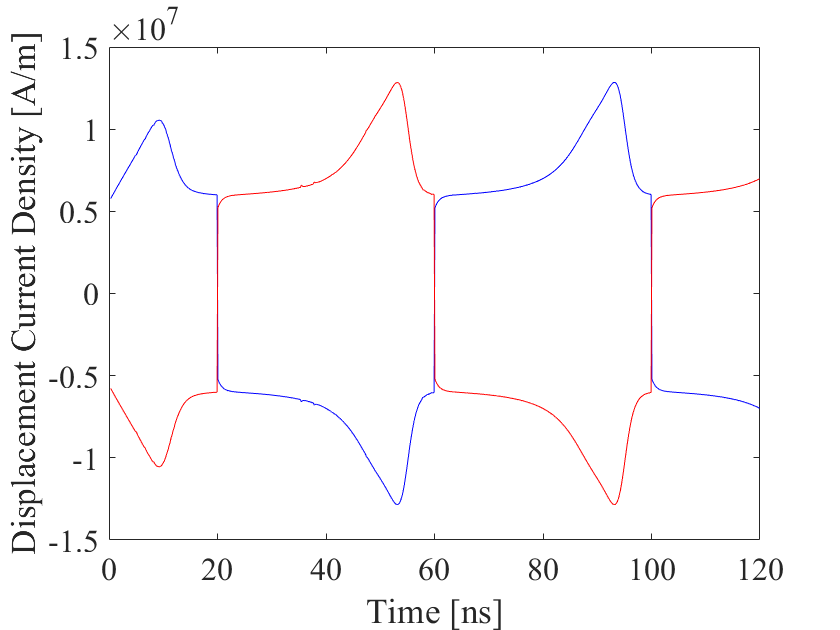} \\
    \end{tabular}
  \end{center}
  \caption{The left plot shows the displacement vector on top (in red) and bottom (in blue) \WL{versus} the variation of the contact bias in $120ns$ (one and a half periods). The right plot shows the displacement current density against time.}
  \label{f:displacement}
\end{figure}

\begin{figure}[hbt!]
  \begin{center}
    \begin{tabular}{cc}
      \includegraphics[scale=0.062]{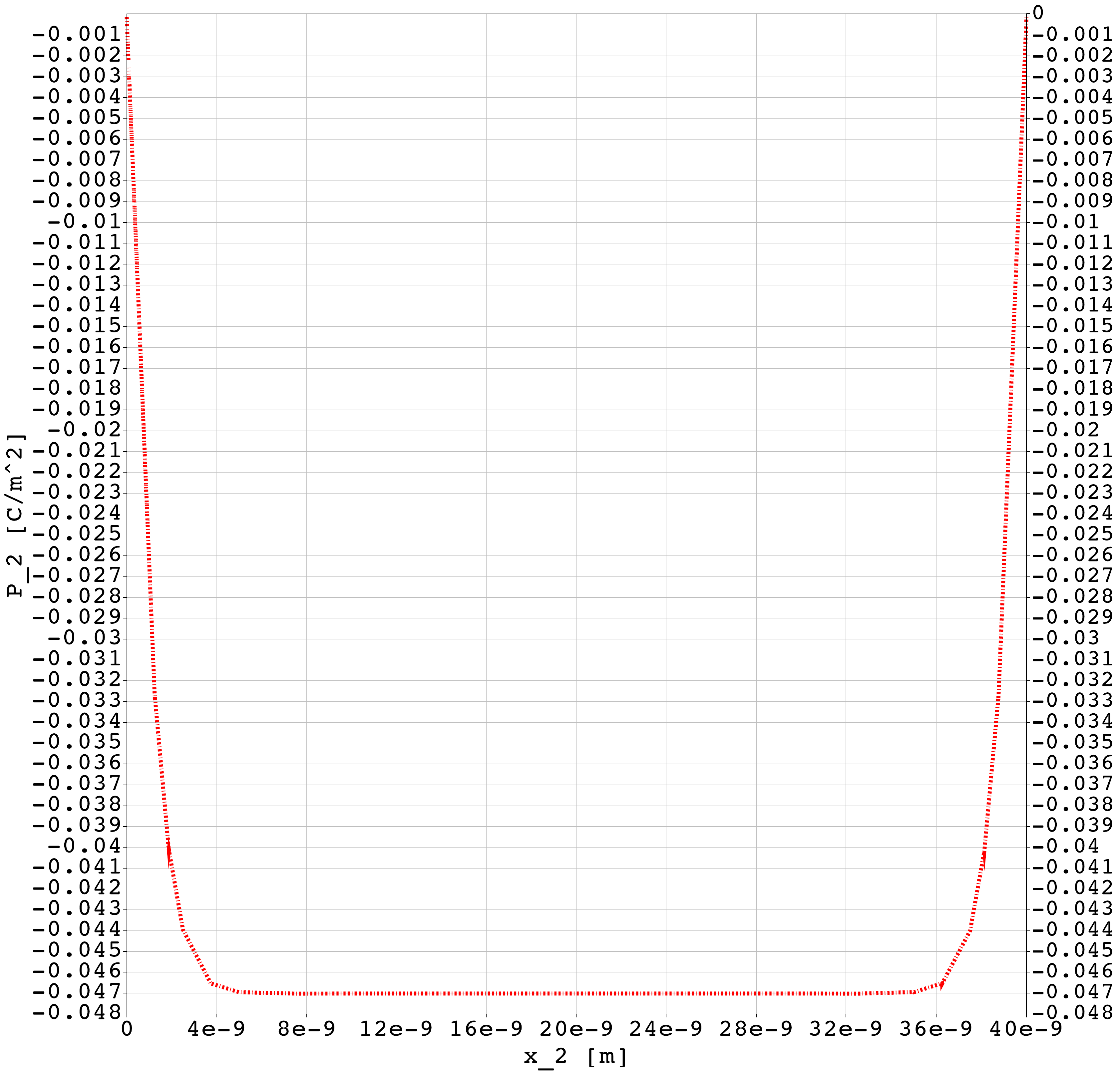}  & \includegraphics[scale=0.062]{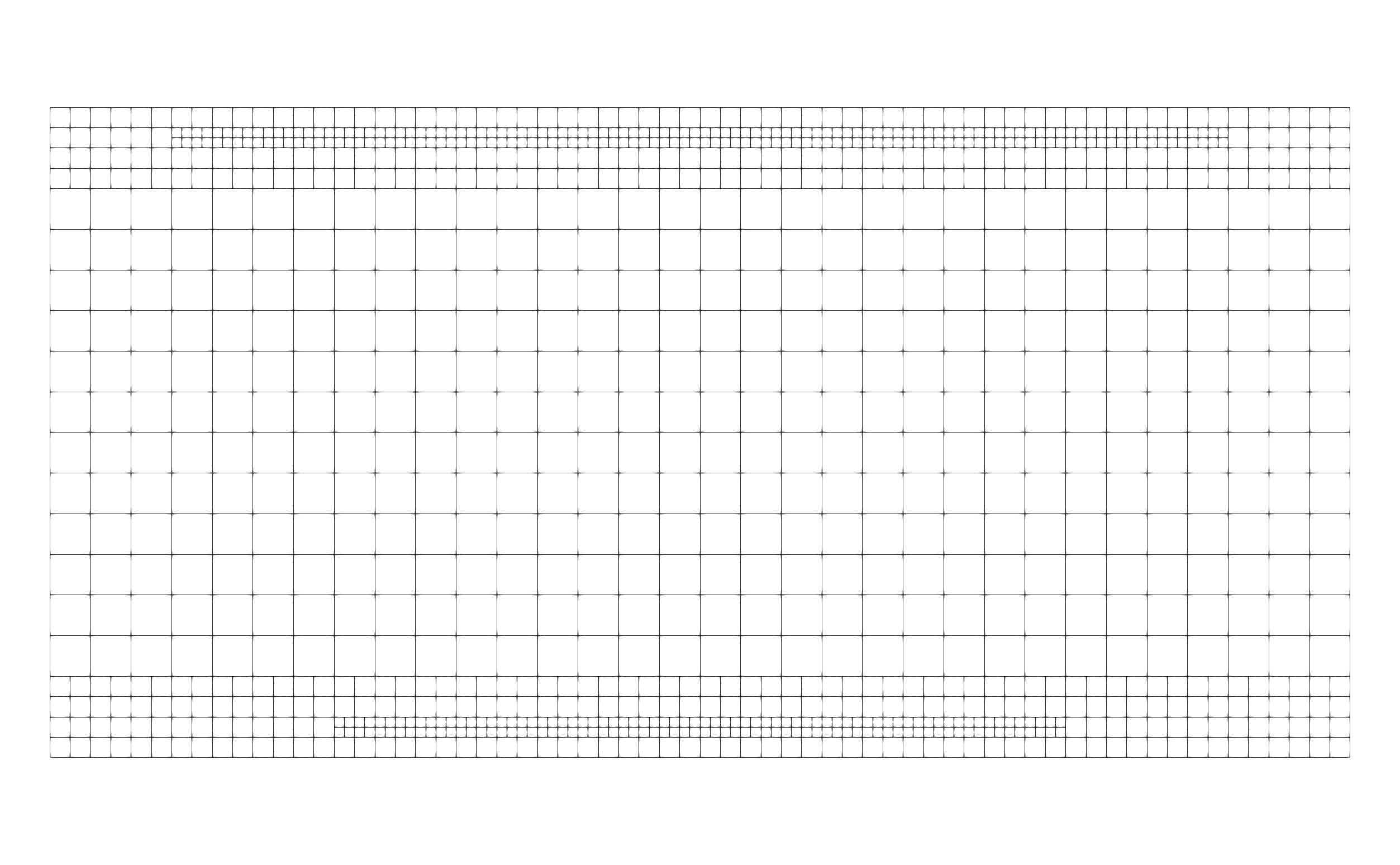}  \\
      \includegraphics[scale=0.062]{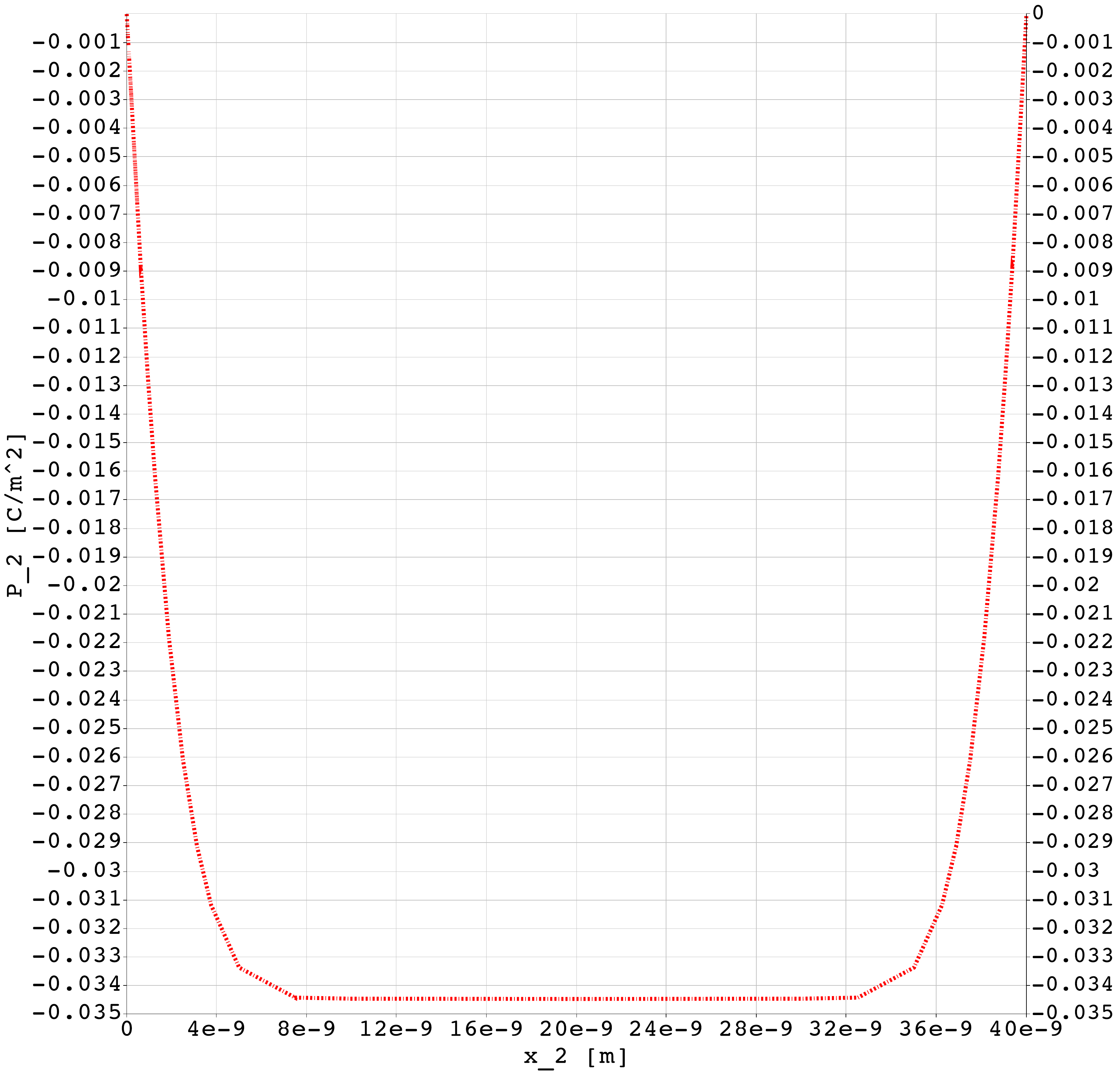}  & \includegraphics[scale=0.062]{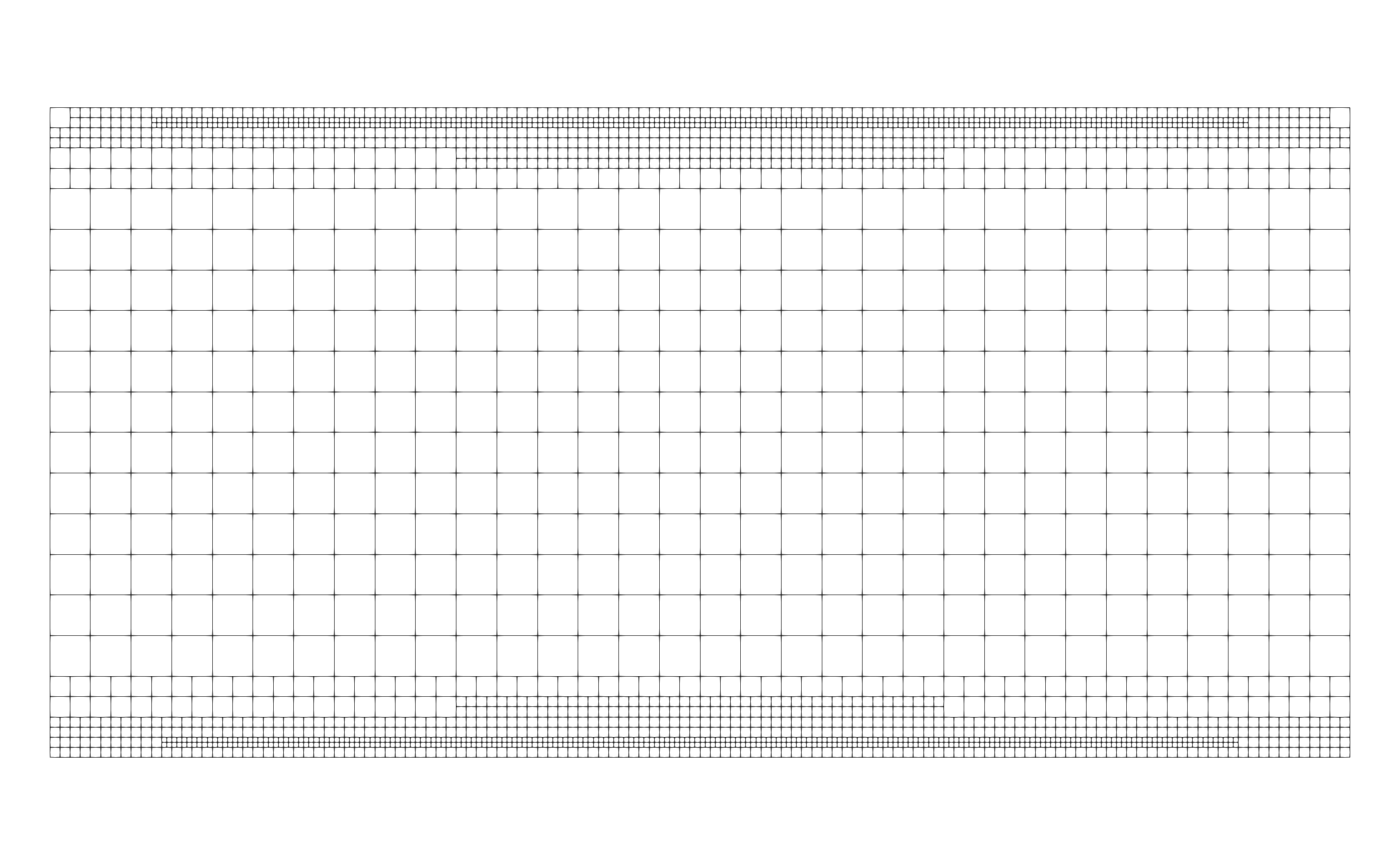}  \\
      \includegraphics[scale=0.062]{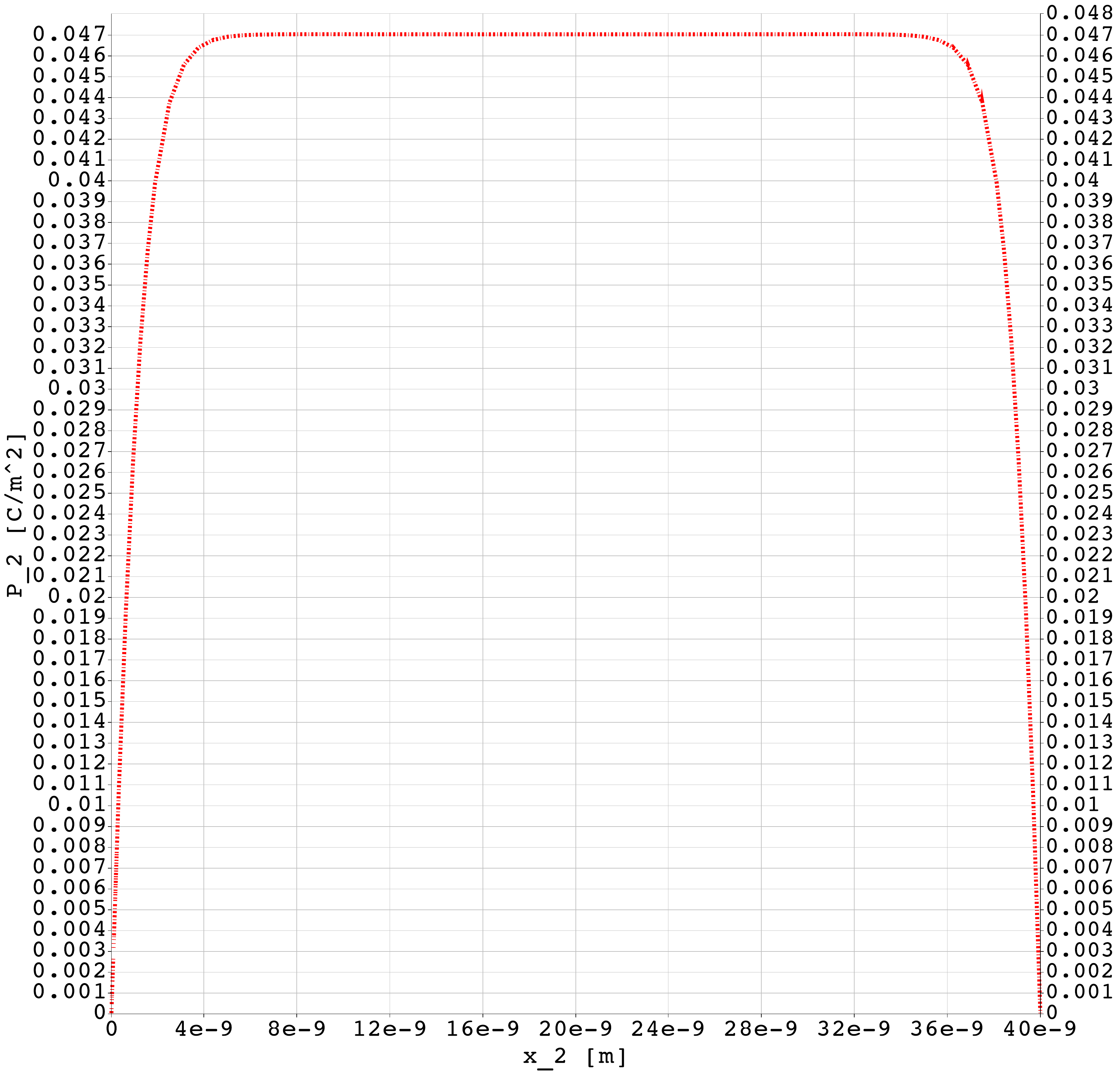}  & \includegraphics[scale=0.062]{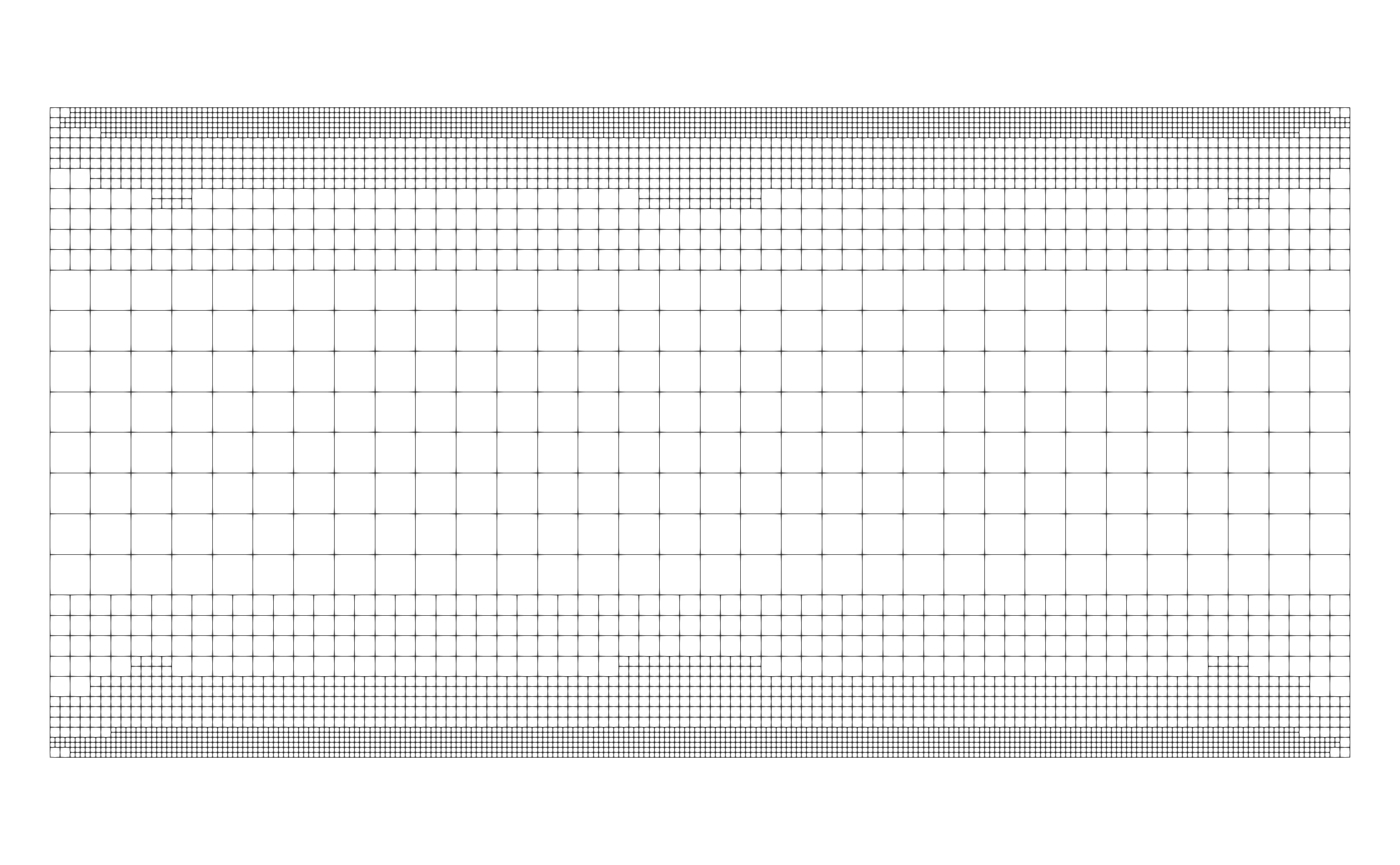}  \\
      \includegraphics[scale=0.062]{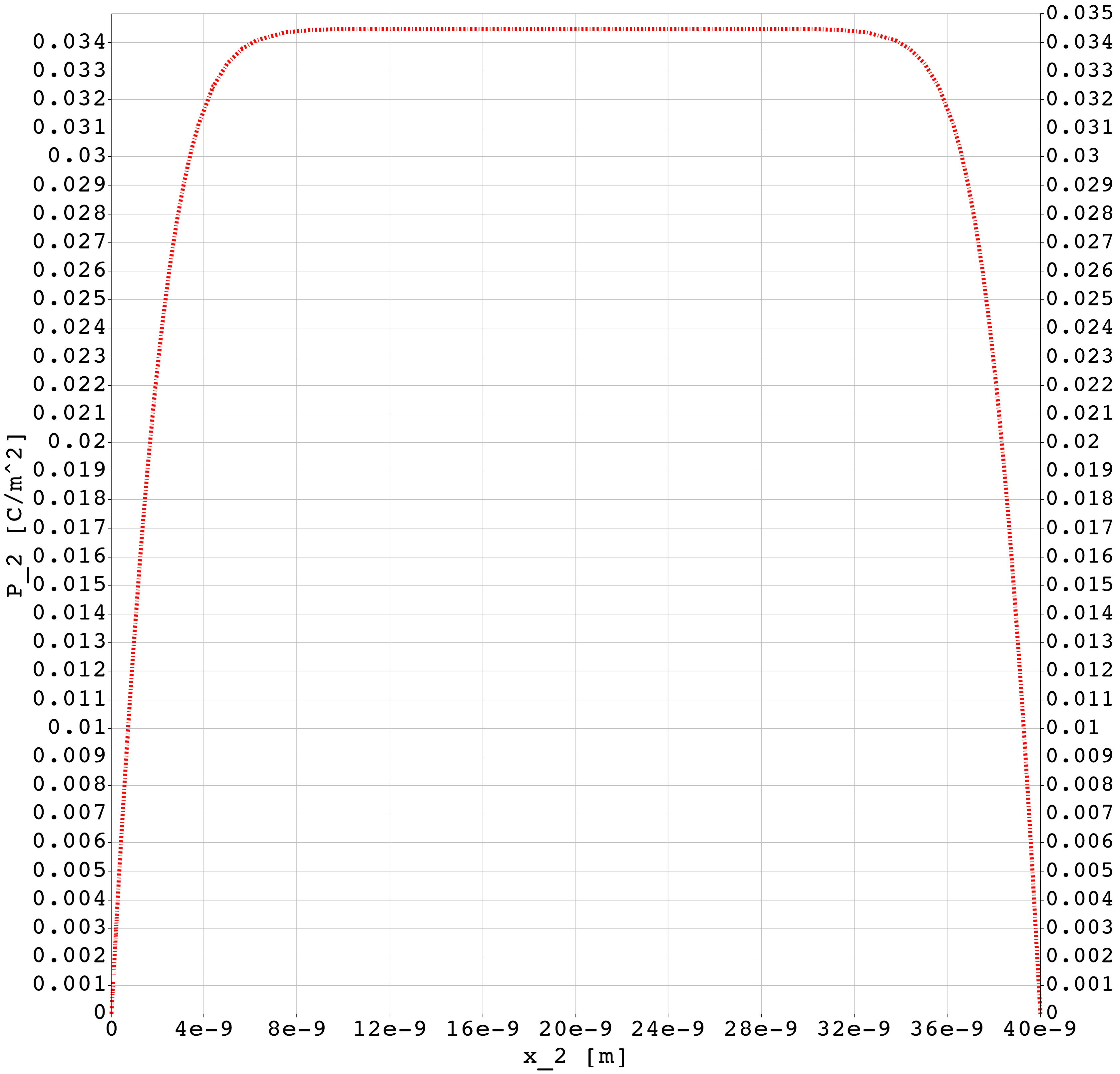}  & \includegraphics[scale=0.062]{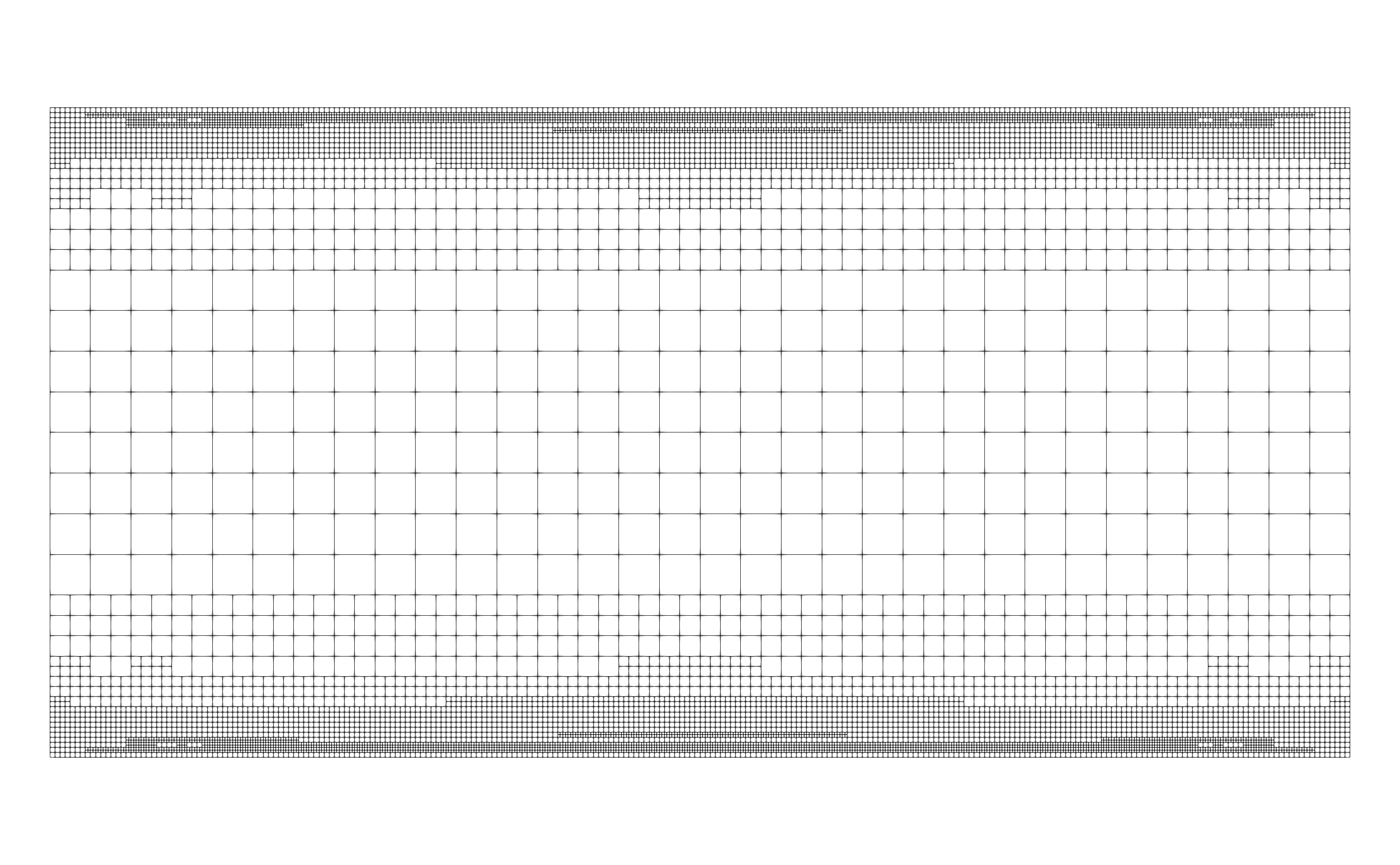}  \\
      \includegraphics[scale=0.062]{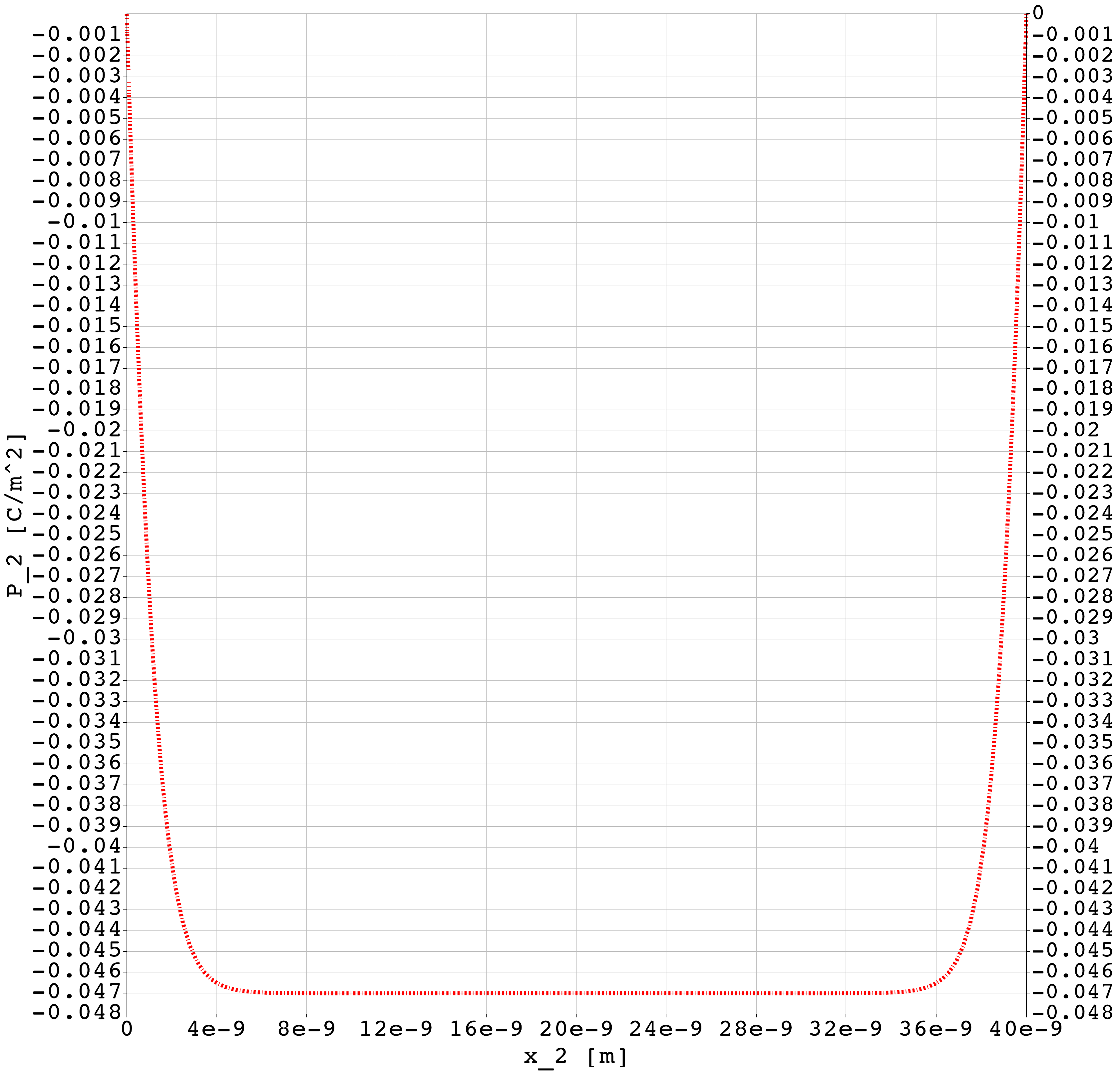} & \includegraphics[scale=0.062]{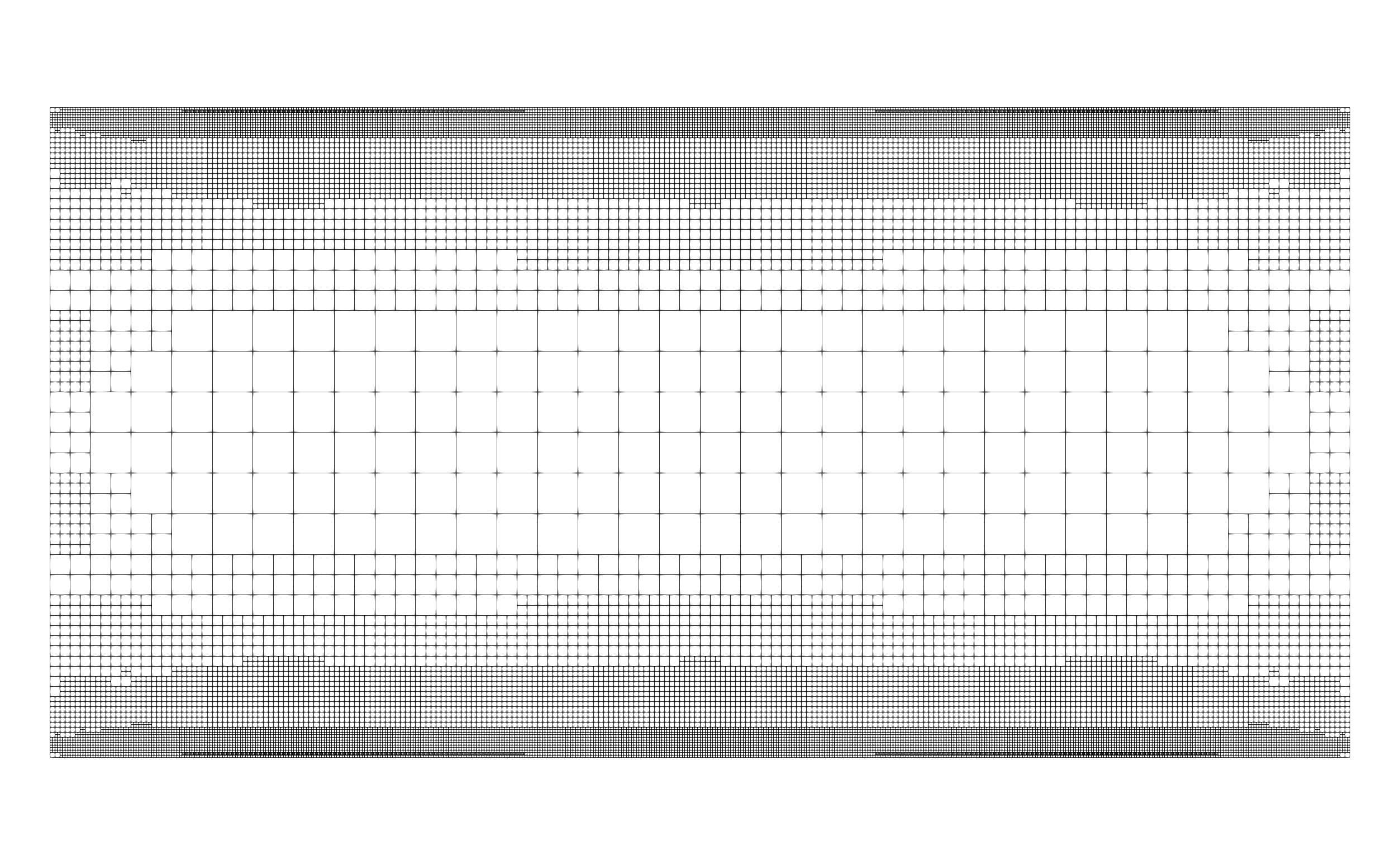} \\
    \end{tabular}
  \end{center}
  \caption{An integration test: (left column) the $x_2$ component of the approximated polarization field (\ie $\widehat{\vecP}^n_h$) along $x_1 = 40ns$ at time $t=20j\, ns$ with $j=1,\ldots,5$ (from top to bottom) and (right column) the corresponding background subdivisions.}
  \label{f:polarization-zero}
\end{figure}
\begin{figure}[hbt!]
  \begin{center}
    \begin{tabular}{cc}
      \includegraphics[scale=0.062]{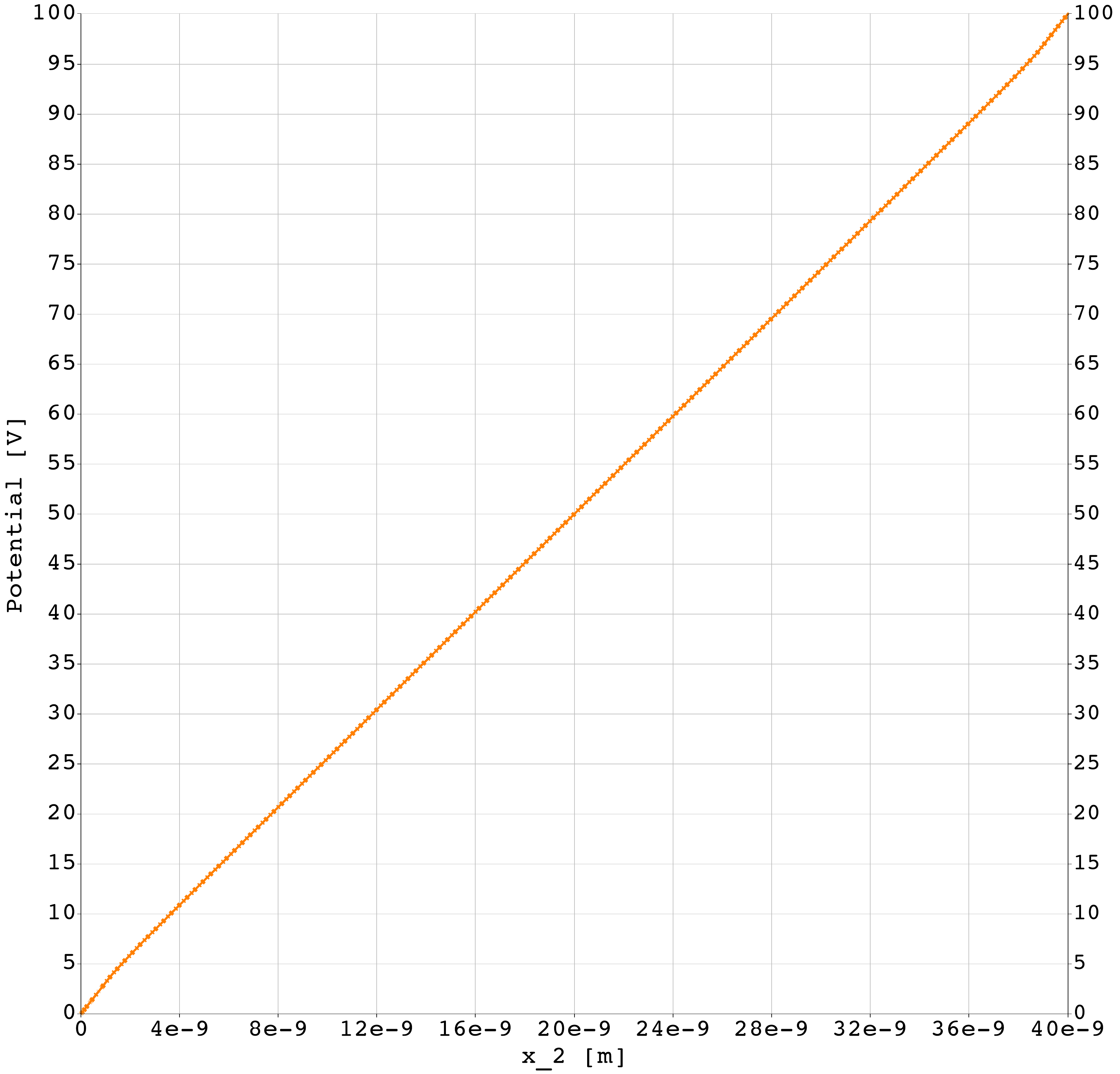}  & \includegraphics[scale=0.062]{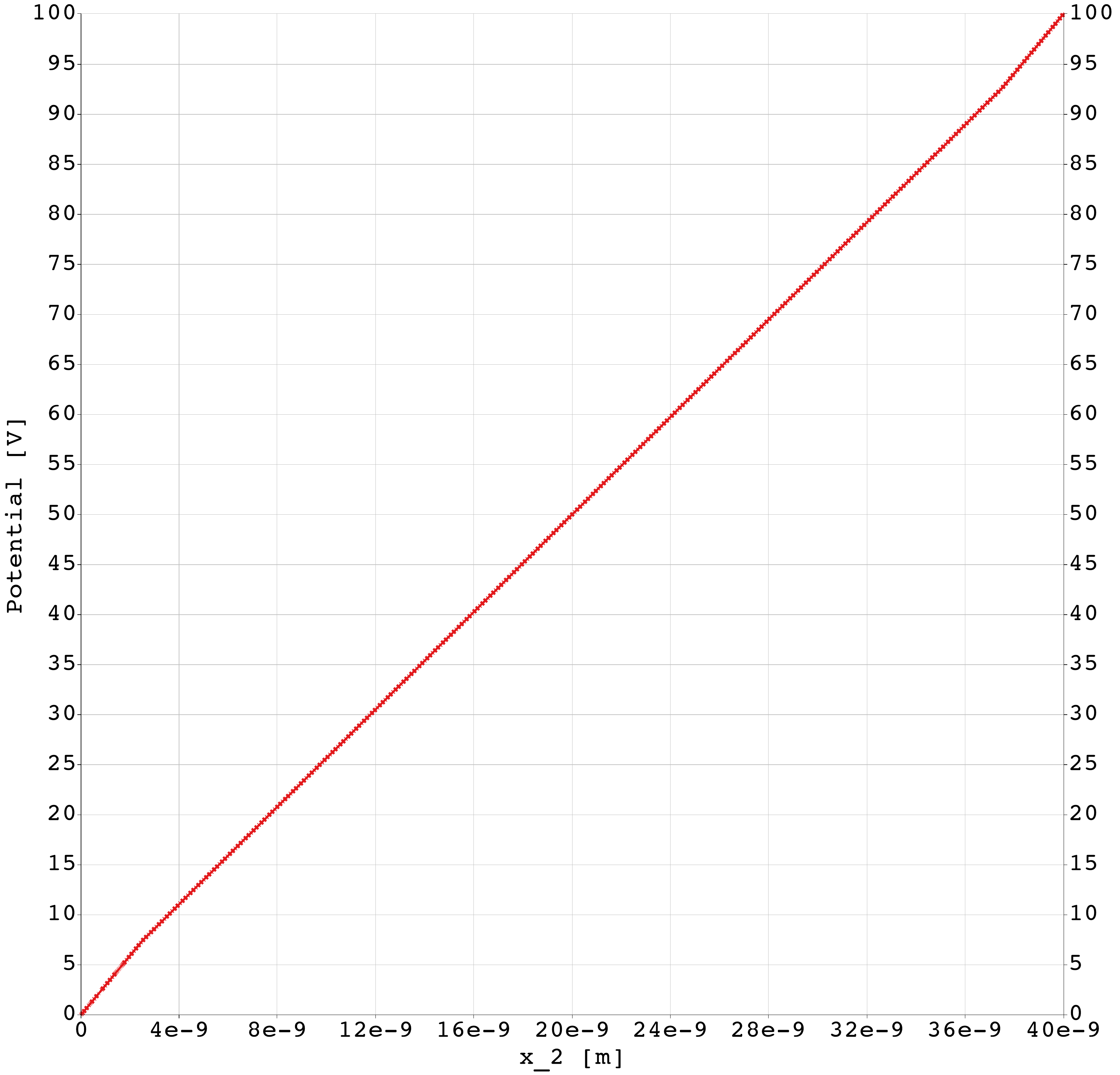}  \\
      \includegraphics[scale=0.062]{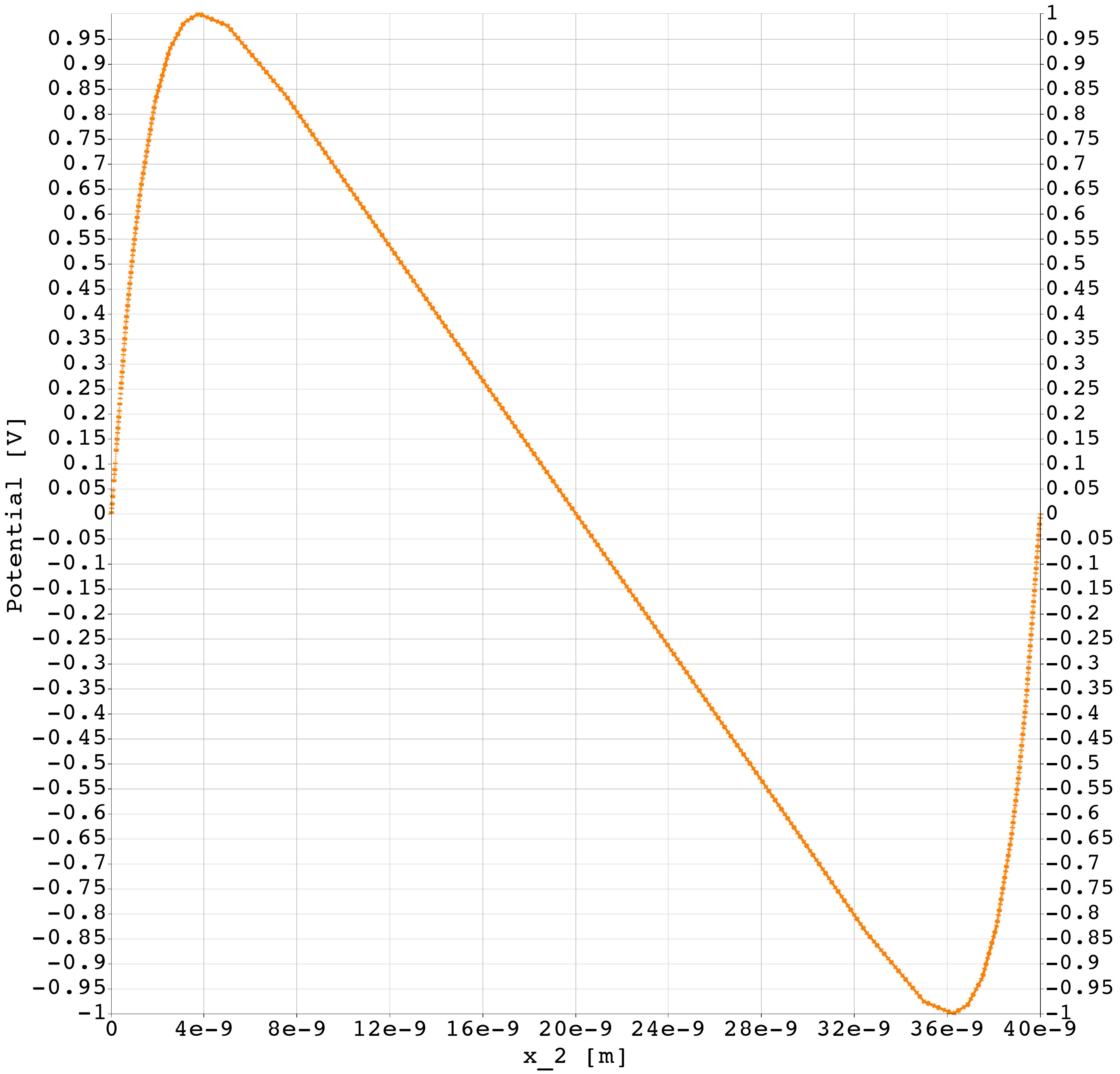}  & \includegraphics[scale=0.062]{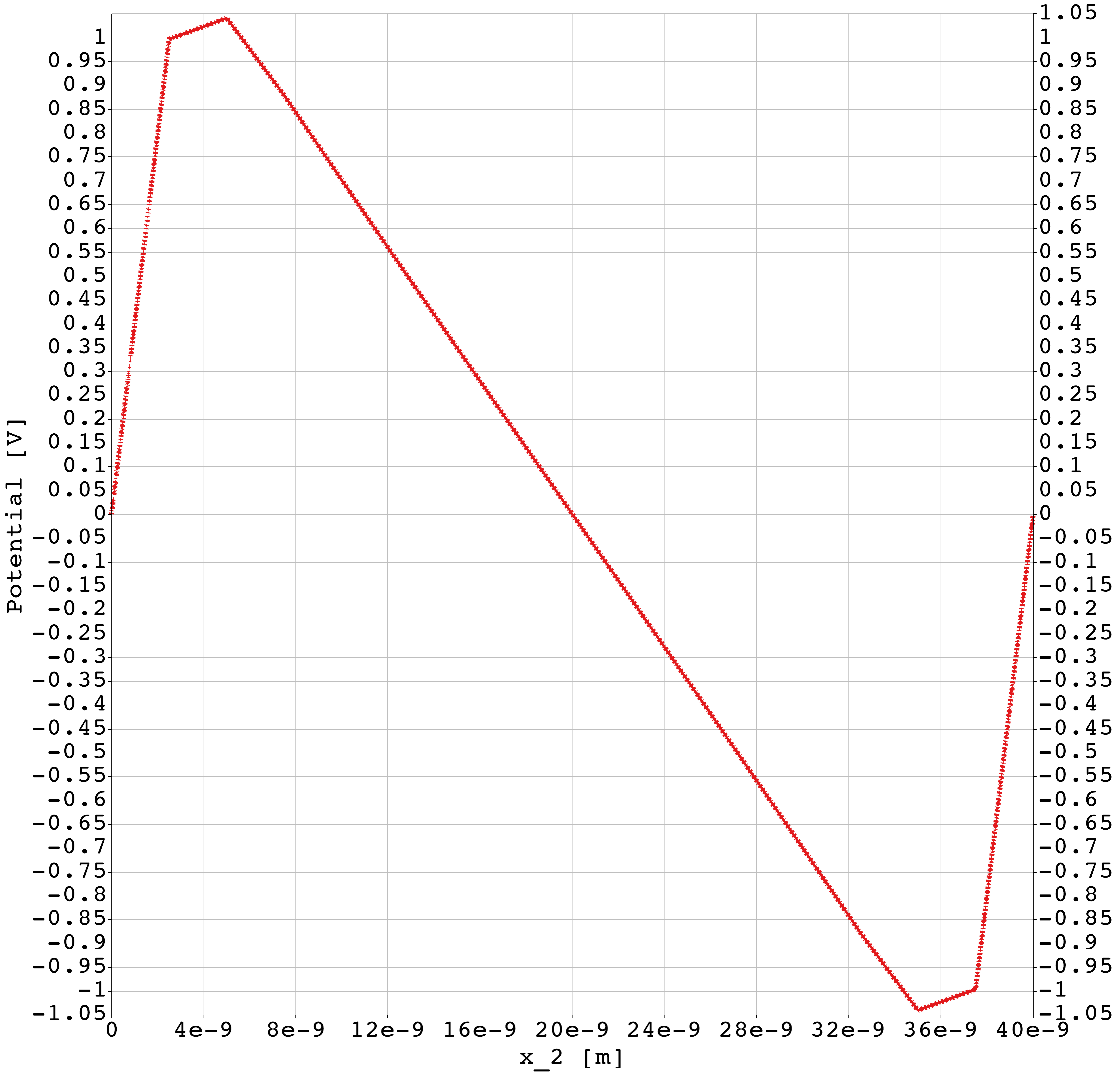}  \\
      \includegraphics[scale=0.062]{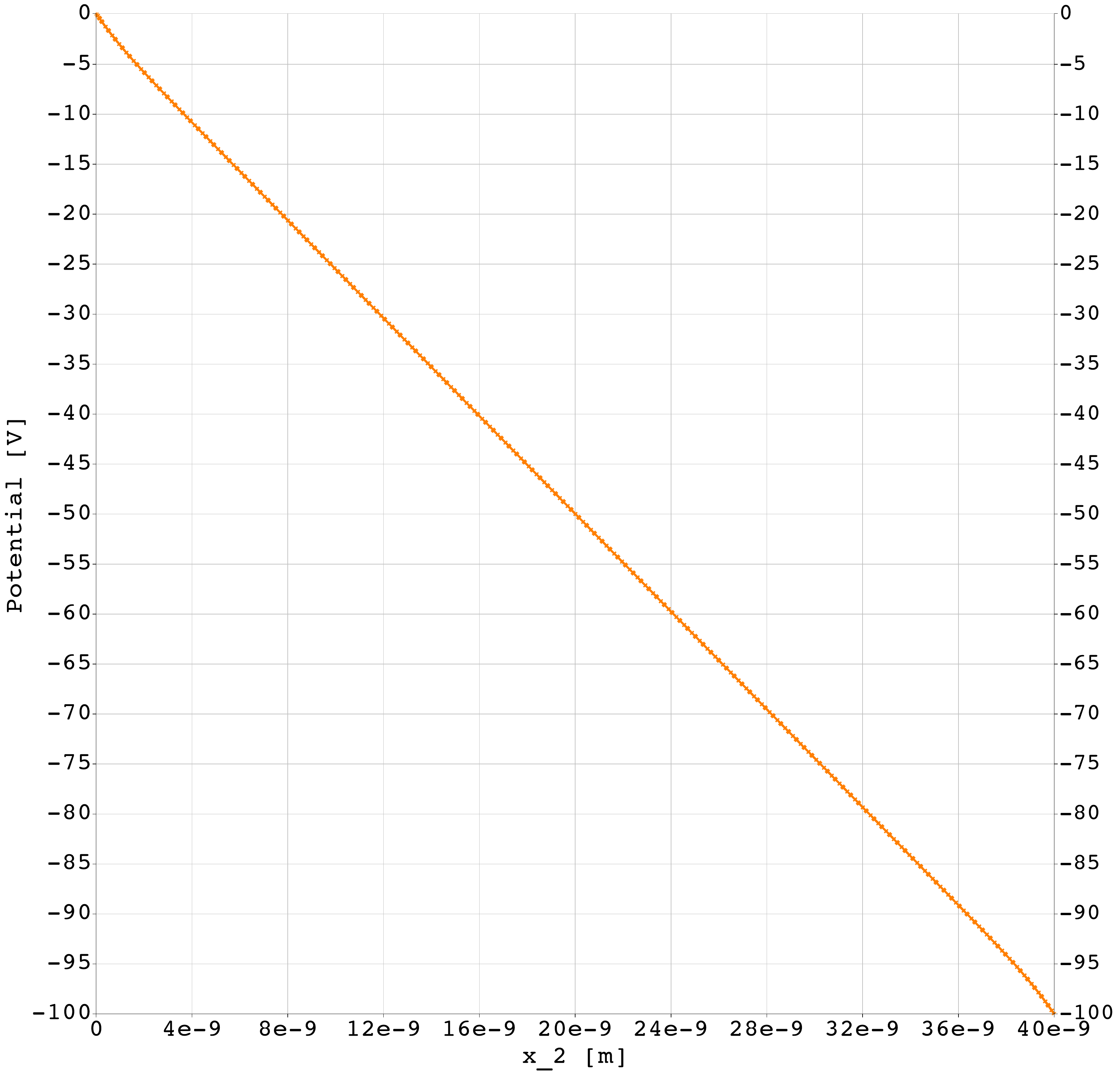}  & \includegraphics[scale=0.062]{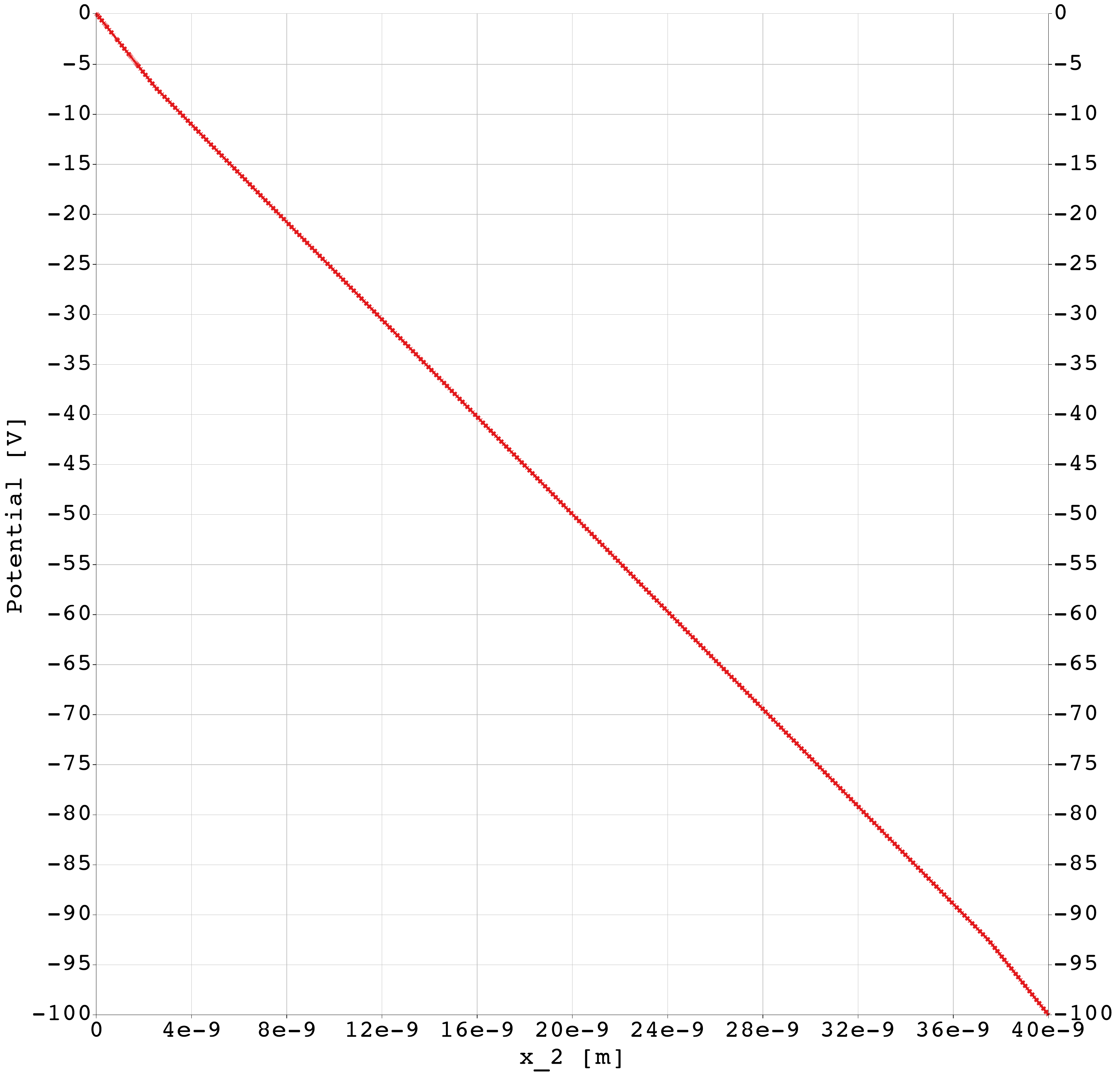}  \\
      \includegraphics[scale=0.062]{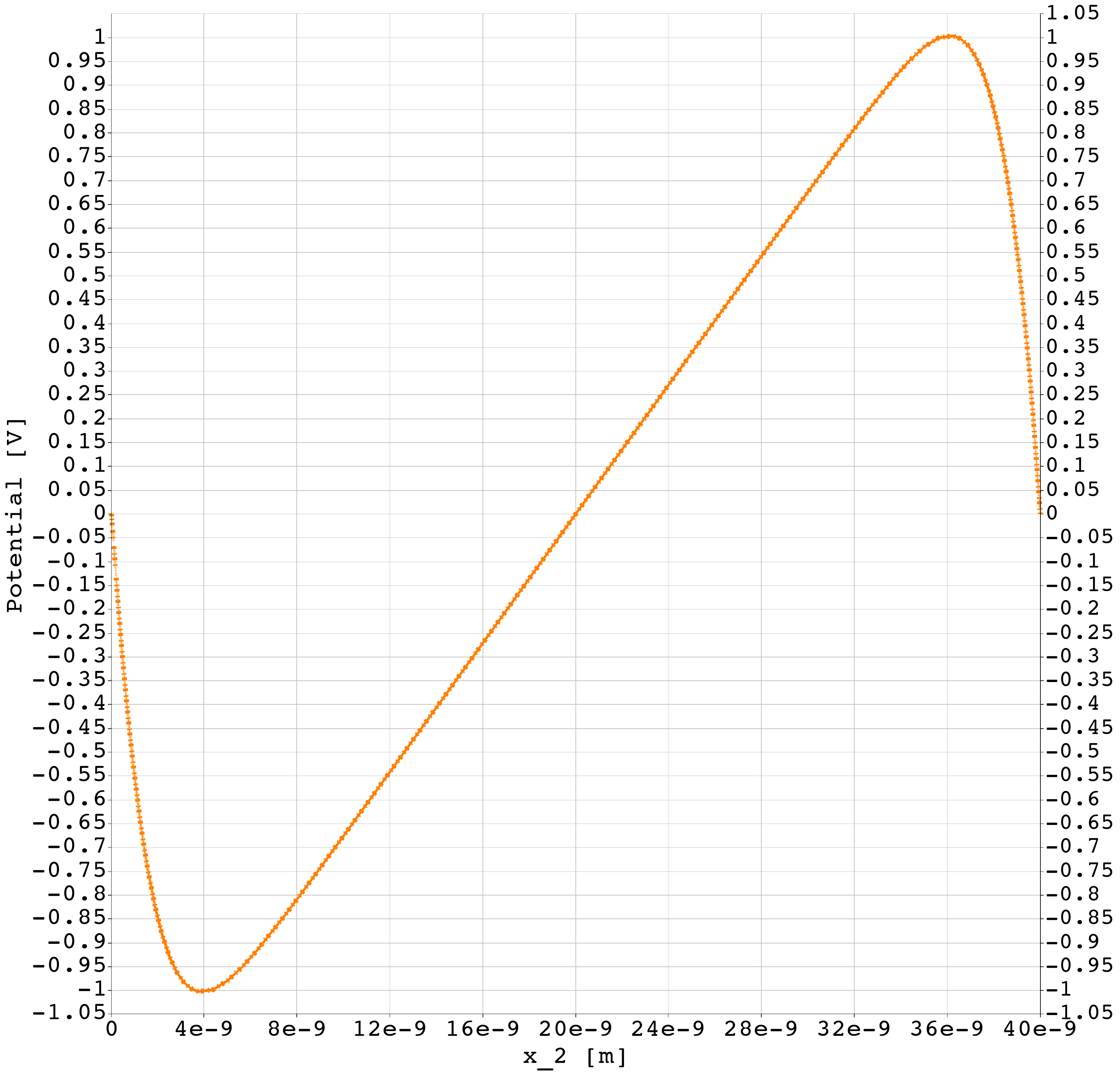}  & \includegraphics[scale=0.062]{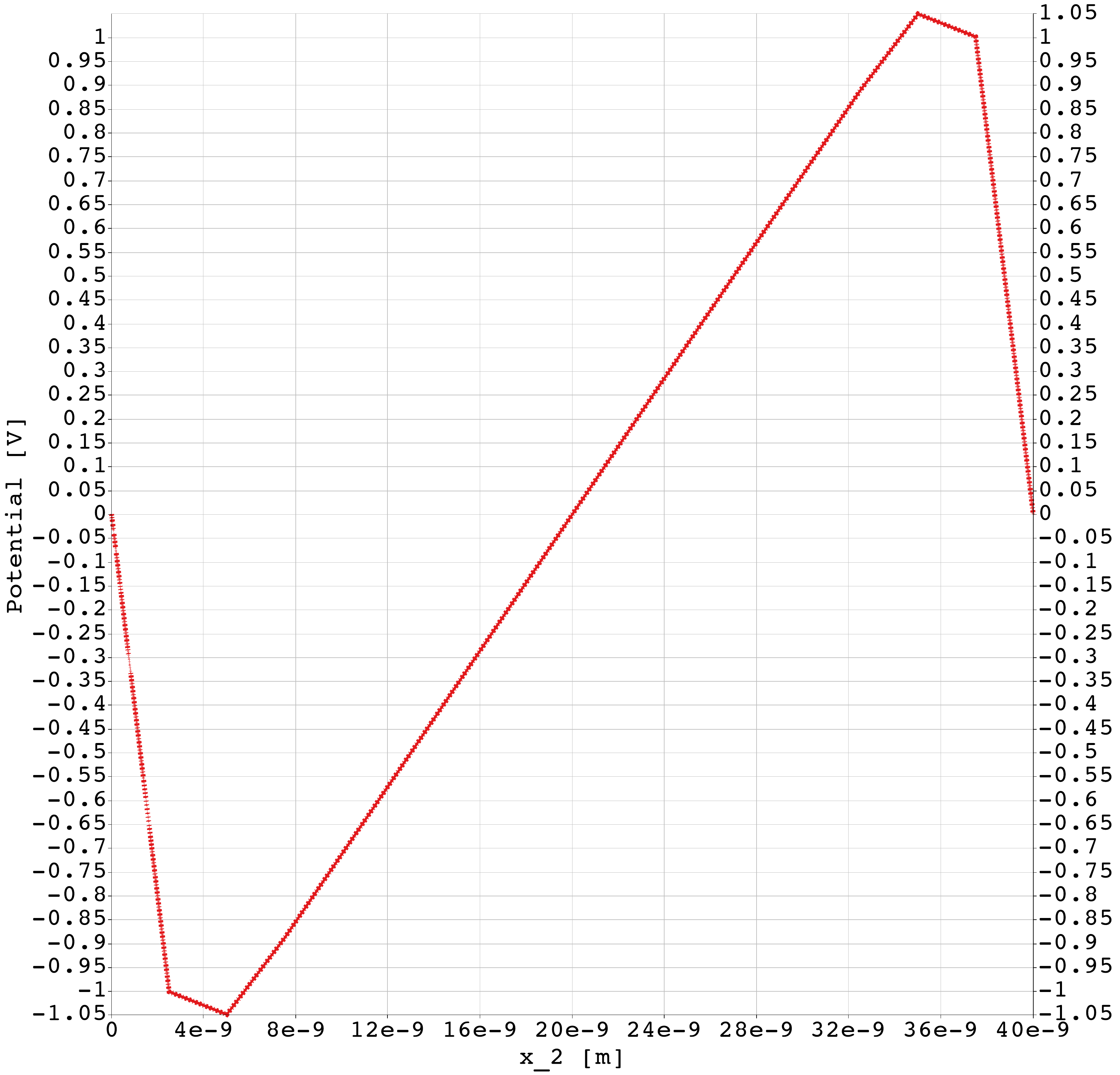}  \\
      \includegraphics[scale=0.062]{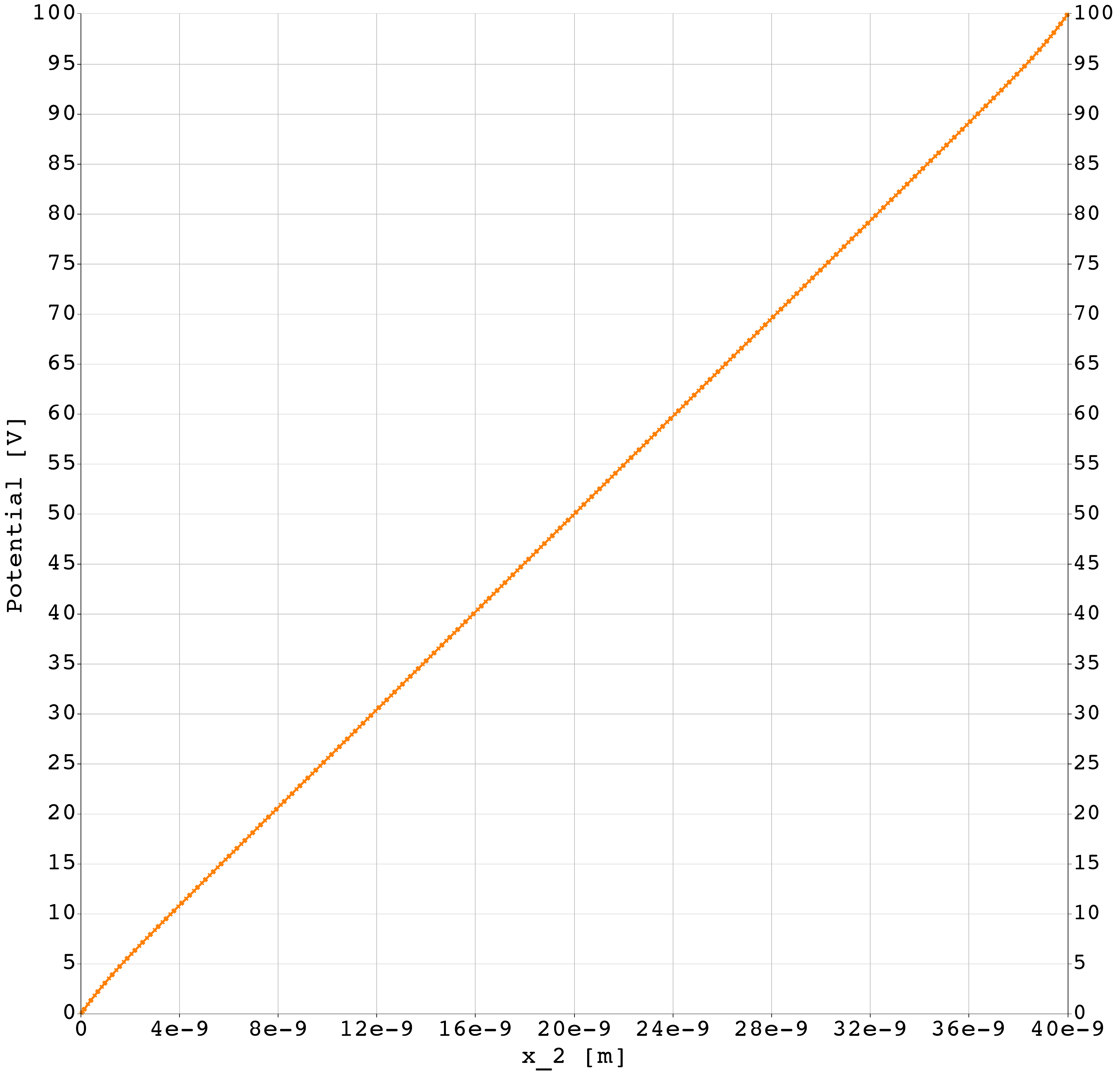} & \includegraphics[scale=0.062]{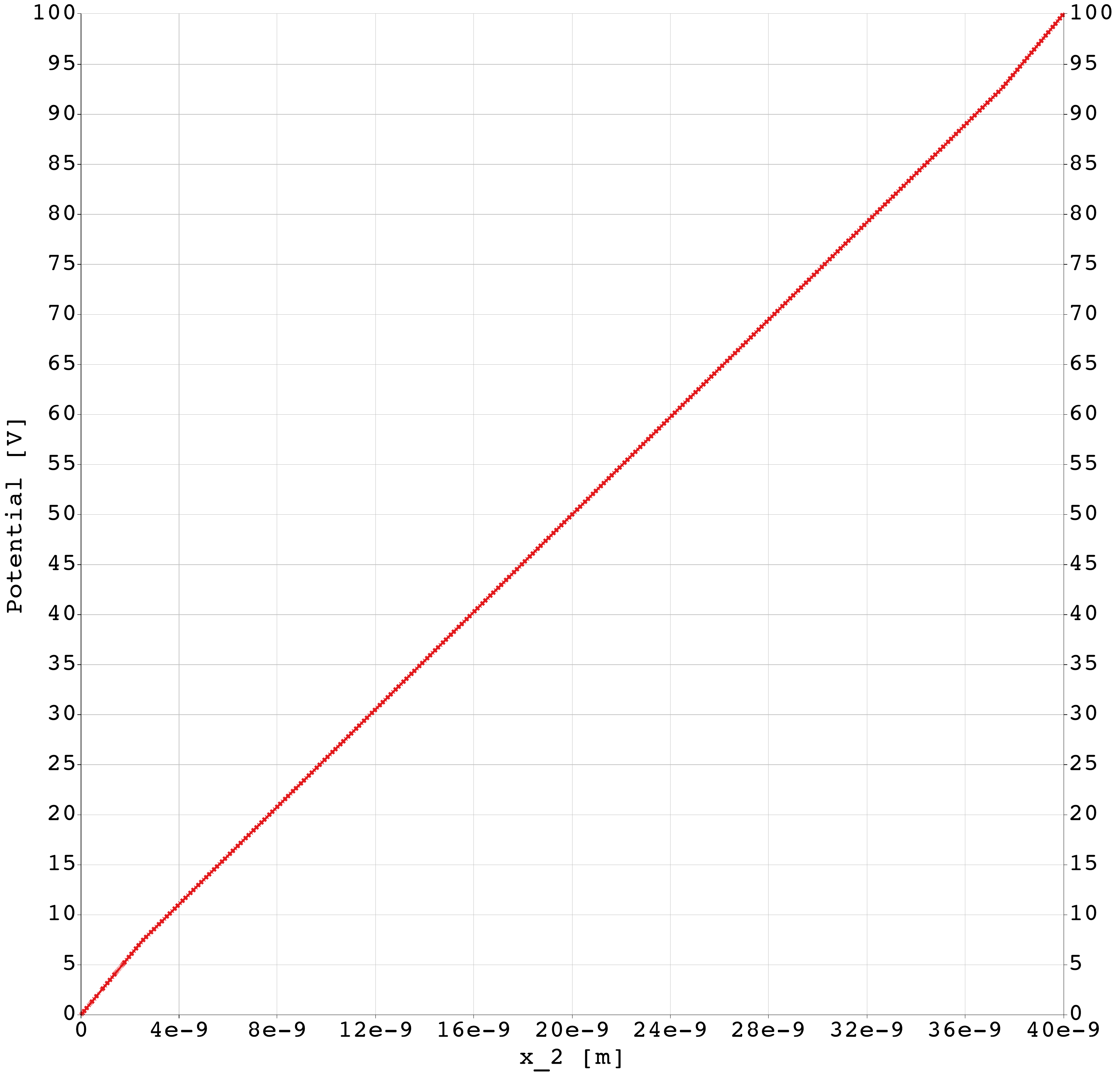} \\
    \end{tabular}
  \end{center}
  \caption{An integration test:  the approximated potential (\ie $\widehat{V}^n_h$) along $x_1 = 40ns$ at time $t=20j\, ns$ with $j=1,\ldots,5$ (from top to bottom) using the adaptive grids from Figure~\ref{f:polarization-nonzero} (left column) and the corase grid from Figure~\ref{f:mesh} (right column).}
  \label{f:potential}
\end{figure}

\section{Conclusion}\label{s:conclusion}
We proposed a modified energy form $I(\vecP)$ based on Landau-Khalatnikov equation \eqref{e:ferroelectric} and show that this energy form has at least one minimizer. We also show that the weak form of \eqref{e:ferroelectric} admits at least one solution if the initial polarization field is in $[H^1(\Omega)\cap L^\infty(\Omega)]^d$ and charge distribution is in $H^1(0,T;L^\infty(\Omega))$. The solution is unique if the $L^2(\Omega)$-norm of $\Delta \vecP_0$ is finite. We designed an energy stable time-stepping scheme and we use an HDG scheme to discretize \eqref{e:ferroelectric} in space. The numerical simulations verifies that the $L^2(\Omega)$-error for the primal variables converges in the rate $O(h^{k+1}+\tau)$, with $h$ and $\tau$ denoting the sizes in space grid and time grid, respectively. The future work includes the convergence analysis for the \ac{HDG} scheme \eqref{e:full-system}--\eqref{e:transmission-discrete} as well as the analysis from the discrete solution to the minimizer of $I(\vecP)$.

%\section*{Acknowledgement}
%The authors would like to thank Dr. Luca Heltai and Dr. Milan Pe{\v{s}}i{\'{c}} for the helpful discussions on the model. The authors would also like to thank anonymous referees for their helpful comments and suggestions to improve this manuscript.
%
%\section*{Conflict of interest}
%The authors declare that they have no conflict of interest.
%
%\section*{Funding}
There are no funders to report for this submission.

% \nocite{*}% Show all bib entries - both cited and uncited; comment this line to view only cited bib entries;
\bibliographystyle{plain}      % mathematics and physical sciences

\bibliography{reference}
\end{document}